\documentclass[a4paper,12pt]{amsart}
\usepackage{amsfonts}
\usepackage{amsmath,array}
\usepackage{amssymb}
\usepackage{mathrsfs}
\usepackage[colorlinks]{hyperref}
 \usepackage{graphicx}
\usepackage{enumerate}
\usepackage[usenames]{color}
\newtheorem{theorem}{Theorem}[section]
\newtheorem{cor}[theorem]{Corollary}

\numberwithin{equation}{section}
\usepackage[english]{babel} 
\usepackage{blindtext}

\theoremstyle{definition}
\newtheorem{definition}[theorem]{Definition}
\newtheorem{remark}[theorem]{Remark}

 \DeclareMathOperator{\Vol}{Vol}

\addtocontents{toc}{\setcounter{tocdepth}{1}}  
\addtocontents{toc}{\setcounter{tocdepth}{3}}  
\begin{document}
 \title[Fourier multipliers on quantum Euclidean spaces]{Fourier multipliers and their applications to PDE on the quantum Euclidean space}

\author[Michael Ruzhansky]{Michael Ruzhansky}
\address{
 Michael Ruzhansky:
  \endgraf
 Department of Mathematics: Analysis, Logic and Discrete Mathematics,
  \endgraf
 Ghent University, Ghent,
 \endgraf
  Belgium 
  \endgraf
  and 
 \endgraf
 School of Mathematical Sciences, Queen Mary University of London, London,
 \endgraf
 UK  
 \endgraf
  {\it E-mail address} {\rm michael.ruzhansky@ugent.be}
  }

\author[Serikbol Shaimardan]{Serikbol Shaimardan}
\address{
  Serikbol Shaimardan:
  \endgraf
  \endgraf
Department of Mathematics: Analysis, Logic and Discrete Mathematics
  \endgraf
 Ghent University, Ghent,
 \endgraf
  Belgium
  \endgraf
   and 
    \endgraf
  Institute of Mathematics and Mathematical Modeling, 050010, Almaty, 
  \endgraf
  Kazakhstan 
  \endgraf
  {\it E-mail address} {\rm shaimardan.serik@gmail.com} 
  }

 \author[Kanat Tulenov]{Kanat Tulenov}
\address{
  Kanat Tulenov:
   \endgraf
  Department of Mathematics: Analysis, Logic and Discrete Mathematics
  \endgraf
 Ghent University, Ghent,
 \endgraf
  Belgium
  \endgraf
   and 
   \endgraf
  Institute of Mathematics and Mathematical Modeling, 050010, Almaty, 
  \endgraf
  Kazakhstan 
  \endgraf
  {\it E-mail address} {\rm tulenov@math.kz} 
  }

\date{}

\begin{abstract}
In this work, we present some applications of the $L^p$-$L^q$ boundedness of Fourier multipliers to PDEs on the noncommutative (or quantum) Euclidean space. More precisely, we establish $L^p$-$L^q$ norm estimates for solutions of heat, wave, and Schr\"odinger type equations with Caputo fractional derivative in the case $1 < p \leq 2 \leq q < \infty.$ 
Moreover, we obtain well-posedness of nonlinear heat and wave equations on the noncommutative Euclidean space.
\end{abstract}



\maketitle

\tableofcontents
{\section{Introduction}}

 The noncommutative or quantum Euclidean space $\mathbb{R}^d_{\theta}$ defined by an arbitrary skew-symmetric real $d\times d$ matrix $\theta$, which is also known as Moyal space or Phase space \cite{BW, BB, DLMS,  Gro, M, Rieffel, W}, is the main objective in the noncommutative geometry (see \cite{CGRS,  GGSVM, green-book}), since it plays the role of the simplest noncompact manifold. In quantum mechanics, quantum Euclidean spaces serve as models for interactions between classical systems and their corresponding quantum systems, often termed hybrid systems. The operators defining in $\mathbb{R}^d_{\theta},$ known as quantum-classical Weyl operators, have been extensively studied in \cite{BW, DW}. Convolutions in $\mathbb{R}^d_{\theta}$ in the case when $\theta$ is invertable, real, skew-symmetric matrix was first introduced in \cite{W}. The associativity of such convolutions was established there, along with the importance of the Fourier-Weyl transform, which in this work is presented as the Fourier transform on $\mathbb{R}^d_{\theta}.$ The general case for arbitrary $\theta$ is discussed in \cite{DW}. Additionally, \cite{W} establishes numerous analogs of classical results, including a Young-type and Hausdorff-Young inequalities. There are many equivalent constructions of noncommutative Euclidean spaces that appeared in the literature. One of the ways defining them is to construct a von Neumann algebra $\mathbb{R}^d_{\theta}$ with a twisted left-regular representation of $\mathbb{R}^d$ on the Hilbert space $L^2(\mathbb{R}^d).$ More precisely, the von Neumann algebra $\mathbb{R}^{d}_{\theta}$ is generated by the $d$-parameter strongly
continuous unitary family $\{U_{\theta}(t)\}_{t\in \mathbb{R}^{d}}$ satisfying the relation
\begin{equation}\label{weyl-relation}
U_{\theta}(t)U_{\theta}(s)=e^{\frac{1}{2}i(t,\theta s)}U_{\theta}(t+s),\quad t,s\in \mathbb{R}^d,
\end{equation}
where $(\cdot,\cdot)$ means the usual inner product in $\mathbb{R}^d.$ For more details and recent results on this theory, we refer the reader to \cite{GJM, GGSVM, GJP, GJP2, MSX, Mc, Tulenov1, RST2}. We associate the algebra $\mathbb{R}^{d}_{\theta}$ with a semifinite normal trace $\tau_{\theta},$ which coincides with the Lebesgue integral when $\theta$ is zero matrix, thereby providing a concept of integration
for our von Neuman algebra $\mathbb{R}^{d}_{\theta}.$ If we are given a von Neuman algebra with a trace, then there is a theory to define noncommutative $L^p$ spaces \cite{DPS, FK, LSZ, PXu}.
As a particular case of the theory of $L^p$ spaces associated with semifinite traces on von Neumann algebras, we can construct $L^p$ spaces over the noncommutative Euclidean space, which we denote as $L^{p}(\mathbb{R}^{d}_{\theta}).$ These spaces reduce to the classical 
$L^p$ spaces on Euclidean space when $\theta=0$ and when $\theta\neq 0$ are spaces of operators affiliated with the von Neumann algebra $\mathbb{R}^{d}_{\theta}.$

One of the main advantages of this von Neumann algebra $\mathbb{R}^{d}_{\theta}$ is the ability to define various analogs of the classical harmonic analysis tools, including differential operators. These capabilities enable the study of partial differential equations within the noncommutative framework. 
 Recent advances in harmonic analysis on these spaces reveal some of their unique characteristics. Notably, the behaviour of nonlinear partial differential equations in this noncommutative context exhibits surprising features. Significant progress on nonlinear equations in noncommutative spaces has been made by Rosenberg \cite{R}, who explored nonlinear partial differential equations on the noncommutative torus. Later, in \cite{Zhao} Zhao proved the Strichartz estimates and applied it for the Schr\"odinger equations on noncommutative Euclidean space. In \cite{Mc}, McDonald explored fundamental aspects of paradifferential calculus for noncommutative Euclidean spaces and presented applications to nonlinear evolution equations. Additionally, he showed how certain equations become much simpler to analyze in a purely noncommutative framework. We can also mention a recent work by the authors \cite{RST2}, where they examine the Sobolev inequality in noncommutative Euclidean spaces, derive a Gagliardo–Nirenberg type inequality as a consequence, and apply it to prove the global well-posedness of nonlinear partial differential equations in this noncommutative setting. On the other hand, recently the Hardy-Littlewood-Sobolev inequalities were also studied in the quantum phase space by Lafleche in \cite{L}. The quantum phase space is unitarily equivalent to the particular case of our noncommutative Euclidean space when $\theta$ is invertable skew-symmetric matrix. A key distinction in our work is the consideration of cases where the matrix $\theta$ may not be invertible. When $\theta$ is invertible, the $L^p$-spaces on $\mathbb{R}^d_{\theta}$ are isomorphic to the Schatten classes, and in this setting, there exist well-defined embeddings between $L^p$-spaces for different $p.$  However, in the general such inclusions may no longer hold, leading to significant differences in the analysis. This distinction plays a crucial role in the inequalities and results we derive for PDEs.

 In this work we consider general case. These studies provide strong motivation for further exploration of nonlinear partial differential equations within the noncommutative Euclidean space and offer additional insight into this field.
 
In our recent paper \cite{Tulenov1}, we studied Fourier multipliers on quantum Euclidean spaces and obtained a quantum analogue of the classical H\"ormander Fourier multiplier theorem settled in \cite[Theorem 1.11]{Hor}. Indeed, we established $L^p$-$L^q$ boundedness of the Fourier multipliers on quantum Euclidean spaces for the case $1 < p \leq  2\leq q <\infty.$ In this note, we present many applications of this result to nonlinear PDEs on the quantum Euclidean space. Namely, we derive $L^p$-$L^q$ norm estimates for $\mathcal{L}$-heat, $\mathcal{L}$-wave as well as Schr\"odinger type equations within the range $1 < p \leq 2 \leq q < +\infty$. We also explore the existence, uniqueness, and time-decay behavior of these equations. Similar investigations can be found from the papers \cite{AR, CKRT, SJR1, SJR2, KR}  for several types of groups and spaces. {\color{red} When $\theta=0$ and $\mathcal{L}= (-\Delta_{\theta})^\frac{s}{2}, \quad 0<s\leq 2,$ similar problems were studied in \cite{KSZ,XSM2,QH, XSM} and references therein.}

\;\;\;\;\;\;\;
\section{Preliminaries}
\subsection{Noncommutative (NC) Euclidean space $ \mathbb{R}^{d}_{\theta} $} \label{NC Euclidean space}

For a detailed discussion of the noncommutative (NC) Euclidean space $ \mathbb{R}^{d}_{\theta} $ and further information, we refer the reader to \cite{GJM}, \cite{GGSVM}, \cite{MSX}, and \cite{Mc}.

Let $ H $ be a Hilbert space, and denote by $ B(H) $ the algebra of all bounded linear operators on $ H $. We denote by $ L^p(\mathbb{R}^d) $ (for $ 1 \leq p < \infty $) the set of all pointwise almost-everywhere equivalence classes of $ p $-integrable functions, while $ L^{\infty}(\mathbb{R}^d) $ is the space of essentially bounded functions on the Euclidean space $ \mathbb{R}^d $. For $ 1 \leq p \leq \infty $, the weak-$ L^p $ space on $ \mathbb{R}^d $ consists of all complex-valued measurable functions on $ \mathbb{R}^d $ for which the following quasinorm is finite:
\begin{equation}
\|f\|_{L^{p,\infty}(\mathbb{R}^d)} = \sup\limits_{t>0} t^\frac{1}{p} f^*(t),
\end{equation}
where $ f^* $ is the decreasing rearrangement of $ f $. These spaces are quasi-Banach spaces; for more information, see \cite{Grafakos}.

Given an integer $ d \geq 1 $, fix an skew-symmetric $ \mathbb{R} $-valued $ d \times d $ matrix $ \theta = \{\theta_{j,k}\}_{1 \leq j,k \leq d} $.

\begin{definition}\label{NC_E_space1}
Define $ \mathbb{R}^{d}_{\theta} $ as the von Neumann algebra generated by the $ d $-parameter strongly continuous unitary family $ \{U_{\theta}(t)\}_{t \in \mathbb{R}^{d}} $ satisfying the relation
\begin{equation}\label{weyl-relation}
U_{\theta}(t)U_{\theta}(s) = e^{\frac{1}{2}i(t,\theta s)} U_{\theta}(t+s), \quad t,s \in \mathbb{R}^d,
\end{equation}
where $ (\cdot,\cdot) $ denotes the usual inner product in $ \mathbb{R}^d $.
\end{definition}

The above relation is known as the Weyl form of the canonical commutation relation. While $ \mathbb{R}^{d}_{\theta} $ can be defined abstractly using operator-theoretic methods as in \cite{GGSVM}, it can also be described concretely as generated by a specific family of operators on the Hilbert space $ L^2(\mathbb{R}^d) $.

\begin{definition} \cite[Definition 2.1]{Mc}\label{NC_E_space2}
For $ t \in \mathbb{R}^d $, define $ U_{\theta}(t) $ as the operator on $ L^{2}(\mathbb{R}^{d}) $ given by
\[
(U_{\theta}(t)\xi)(s) = e^{i(t,s)}\xi(s-\frac{1}{2}\theta t), \quad \xi \in L^{2}(\mathbb{R}^{d}), \, t,s \in \mathbb{R}^d.
\]
It can be shown that the family $ \{U_{\theta}(t)\}_{t \in \mathbb{R}^{d}} $ is strongly continuous and satisfies relation \eqref{weyl-relation}. The von Neumann algebra $ \mathbb{R}^{d}_{\theta} $ is then defined as the weak operator topology closed subalgebra of $ B(L^{2}(\mathbb{R}^{d})) $ generated by the family $ \{U_{\theta}(t)\}_{t \in \mathbb{R}^{d}} $, and is referred to as a noncommutative (or quantum) Euclidean space.
\end{definition}

Note that if $ \theta = 0 $, the definitions above reduce to $ L^{\infty}(\mathbb{R}^{d}) $ as the algebra of bounded pointwise multipliers on $ L^{2}(\mathbb{R}^{d}) $. The structure of $ \mathbb{R}^{d}_{\theta} $ is determined by the Stone-von Neumann theorem, which asserts that if $ \det(\theta) \neq 0 $, then any two $ C^* $-algebras generated by a strongly continuous unitary family $ \{U_{\theta}(t)\}_{t \in \mathbb{R}^{d}} $ satisfying the Weyl relation \eqref{weyl-relation} are $ * $-isomorphic \cite[Theorem 14.8]{H}. Thus, the algebra of essentially bounded functions on $ \mathbb{R}^{d} $ is a special case of the previous definition.

If $ d \geq 2 $, an arbitrary $ d \times d $ skew-symmetric real matrix can be expressed (up to orthogonal conjugation) as a direct sum of a zero matrix, leading to the $ * $-isomorphism
\begin{equation}\label{direct-sum}
    \mathbb{R}^{d}_{\theta} \cong L^{\infty}(\mathbb{R}^{\dim(\ker(\theta))}) \bar{\otimes} B(L^{2}(\mathbb{R}^{\text{rank}(\theta)/2})),
\end{equation}
where $ \bar{\otimes} $ denotes the von Neumann tensor product \cite{MSX}. In particular, if $ \det(\theta) \neq 0 $, then \eqref{direct-sum} simplifies to
\begin{equation}\label{reduced-direct-sum}
    \mathbb{R}^{d}_{\theta} \cong B(L^{2}(\mathbb{R}^{d/2})).
\end{equation}
These formulas are valid because the rank of a skew-symmetric matrix is always even.

\subsection{Noncommutative integration}

Let $ f \in L^{1}(\mathbb{R}^d) $. Define $ \lambda_{\theta}(f) $ as the operator given by
\begin{equation}\label{def-integration}
\lambda_{\theta}(f)\xi = \int_{\mathbb{R}^d} f(t)U_{\theta}(t)\xi \, dt, \quad \xi \in L^{2}(\mathbb{R}^d).
\end{equation}
This integral converges absolutely in the Bochner sense and defines a bounded linear operator $ \lambda_{\theta}(f): L^{2}(\mathbb{R}^d) \to L^{2}(\mathbb{R}^d) $ such that $ \lambda_{\theta}(f) \in \mathbb{R}^{d}_{\theta} $ (see \cite[Lemma 2.3]{MSX}). 

Denote by $ \mathcal{S}(\mathbb{R}^d) $ the classical Schwartz space on $ \mathbb{R}^d $. For any $ f \in \mathcal{S}(\mathbb{R}^d) $, define the Fourier transform as
$$
\widehat{f}(t) = \int_{\mathbb{R}^d} f(s) e^{-i(t,s)} ds, \quad t \in \mathbb{R}^d.
$$
The noncommutative Schwartz space $\mathcal{S}(\mathbb{R}_{\theta}^d)$  is defined as the image of the classical Schwartz space under $ \lambda_{\theta} $, i.e.,
$$
\mathcal{S}(\mathbb{R}_{\theta}^d) := \{x \in \mathbb{R}^{d}_{\theta} : x = \lambda_{\theta}(f) \text{ for some } f \in \mathcal{S}(\mathbb{R}^d)\}.
$$
We define a topology on $\mathcal{S}(\mathbb{R}_{\theta}^d)$
 as the image of the canonical Fr\'{e}chet topology on $\mathcal{S}(\mathbb{R}^d)$ under $\lambda_{\theta}.$ 
 The topological dual of $\mathcal{S}(\mathbb{R}_{\theta}^d)$ will be denoted by $\mathcal{S}'(\mathbb{R}_{\theta}^d).$
The map $\lambda_{\theta}$ is injective \cite[Subsection 2.2.3]{MSX}, \cite{Mc}, and it can be extended to distributions. 
\subsection{Trace on $ \mathbb{R}^{d}_{\theta} $}

Let $ f \in \mathcal{S}(\mathbb{R}^d) $. We define the functional $ \tau_{\theta} : \mathcal{S}(\mathbb{R}_{\theta}^d) \to \mathbb{C} $ by the expression
\begin{equation}\label{trace-def}
\tau_{\theta}(\lambda_{\theta}(f)) = \tau_{\theta}\left(\int_{\mathbb{R}^d} f(\eta) U_{\theta}(\eta) \, d\eta\right) := f(0).
\end{equation}
This functional $ \tau_{\theta} $ can be extended to a semifinite normal trace on $ \mathbb{R}^{d}_{\theta} $. Furthermore, if $ \theta = 0 $, then $ \tau_{\theta} $ corresponds to the Lebesgue integral under an appropriate isomorphism. In the case where $ \det(\theta) \neq 0 $, $ \tau_{\theta} $ (up to normalization) is equivalent to the operator trace on $ B(L^2(\mathbb{R}^{d/2})) $. For further details, see \cite{GJP}, \cite[Lemma 2.7]{MSX}, \cite[Theorem 2.6]{Mc}.

\subsection{Noncommutative $ L^{p}(\mathbb{R}^{d}_{\theta}) $ and $ L^{p,q}(\mathbb{R}^{d}_{\theta}) $ spaces}

With the definitions provided in the previous sections, $ \mathbb{R}^{d}_{\theta} $ is a semifinite von Neumann algebra with the trace $ \tau_{\theta} $. The pair $ (\mathbb{R}^{d}_{\theta}, \tau_{\theta}) $ is known as a noncommutative measure space. For any $ 1 \leq p < \infty $, we can define the $ L^p $-norm on this space using the Borel functional calculus with the formula:
\[
\|x\|_{L^p(\mathbb{R}^d_{\theta})} = \left(\tau_{\theta}(|x|^p)\right)^{1/p}, \quad x \in \mathbb{R}^{d}_{\theta},
\]
where $ |x| := (x^{*}x)^{1/2} $. The completion of the set $ \{x \in \mathbb{R}^{d}_{\theta} : \|x\|_{p} < \infty\} $ with respect to the norm $ \|\cdot\|_{L^p(\mathbb{R}^d_{\theta})} $ is denoted by $ L^p(\mathbb{R}^d_{\theta}) $. Elements of $ L^p(\mathbb{R}^d_{\theta}) $ are $ \tau_{\theta} $-measurable operators, analogous to the commutative case. These operators are linear, densely defined, closed (potentially unbounded), and affiliated with $ \mathbb{R}^{d}_{\theta} $ satisfying $\tau_{\theta}(\mathbf{1}_{(s, \infty)}(|x|)) < \infty $ for some $ s > 0 $. Here, $ \mathbf{1}_{(s, \infty)}(|x|) $ represents the spectral projection associated with the interval $ (s, \infty) $. The set of all $ \tau_{\theta} $-measurable operators is denoted by $ L^{0}(\mathbb{R}^{d}_{\theta}) $. 

For $ x = x^{\ast} \in L^{0}(\mathbb{R}^{d}_{\theta}) $, the \textit{distribution function} of $ x $ is defined by
\[
n_x(s) = \tau_{\theta}\left(\mathbf{1}_{(s, \infty)}(x)\right), \quad -\infty < s < \infty.
\]
For $ x \in L^{0}(\mathbb{R}^{d}_{\theta}) $, the \textit{generalised singular value function} $ \mu(t, x) $ is defined as
\begin{equation}\label{distribution-function}
\mu(t, x) = \inf\left\{s > 0 : n_{|x|}(s) \leq t\right\}, \quad t > 0.
\end{equation}
This function $ t \mapsto \mu(t, x) $ is decreasing and right-continuous. The norm of $ L^p(\mathbb{R}^d_{\theta}) $ can also be expressed in terms of the generalised singular value function (see \cite[Example 2.4.2, p. 53]{DPS}) as
\begin{equation}\label{mu-norm}
\|x\|_{L^p(\mathbb{R}^d_{\theta})} = \left(\int_{0}^{\infty} \mu^{p}(s, x) \, ds\right)^{1/p}, \quad \text{for } p < \infty, \quad \text{and} \quad \|x\|_{L^{\infty}(\mathbb{R}^d_{\theta})} = \mu(0, x), \quad \text{for } p = \infty.
\end{equation}
The latter equality for $ p = \infty $ was proved in \cite[Lemma 2.3.12.(b), p. 50]{DPS}.

The space $ L^{0}(\mathbb{R}^{d}_{\theta}) $ forms a $ * $-algebra, which can be endowed with a topological $ * $-algebra structure using the neighborhood system $ \{V(\varepsilon, \delta) : \varepsilon, \delta > 0\} $, where 
\[
V(\varepsilon, \delta) = \{x \in L^{0}(\mathbb{R}^{d}_{\theta}) : \mu(\varepsilon, x) \leq \delta\}.
\]
This topology renders $ L^{0}(\mathbb{R}^{d}_{\theta}) $ a metrizable topological $ * $-algebra, and convergence in this topology is referred to as \textit{convergence in measure} \cite{PXu}. Next, we define the noncommutative analogue of weak-$ L^p $ spaces associated with the noncommutative Euclidean space.

\begin{definition}
Let $ 1 \leq p \leq \infty $. The noncommutative weak-$ L^p $ (Lorentz) space $ L^{p,\infty}(\mathbb{R}^{d}_{\theta}) $ is defined by
\[
L^{p,\infty}(\mathbb{R}^{d}_{\theta}) := \{x \in L^{0}(\mathbb{R}^{d}_{\theta}) : \|x\|_{L^{p,\infty}(\mathbb{R}^{d}_{\theta})} < \infty\},
\]
where
\[
\|x\|_{L^{p,\infty}(\mathbb{R}^{d}_{\theta})} := \sup_{t > 0} t^{\frac{1}{p}} \mu(t, x).
\]
\end{definition}

These are noncommutative quasi-Banach spaces. For the general theory of $ L^p $ and Lorentz spaces associated with semifinite von Neumann algebras, see \cite{DPS}, \cite{LSZ}, \cite{PXu}.

\subsection{Fourier transform and Fourier multipliers}

\begin{definition}\label{F-transform}
For any $ x \in \mathcal{S}(\mathbb{R}_{\theta}^d) $, the Fourier transform of $ x $ is defined as the map $ \lambda_{\theta}^{-1} : \mathcal{S}(\mathbb{R}_{\theta}^d) \to \mathcal{S}(\mathbb{R}^d) $ given by
\begin{equation}\label{direct-F-transform}
\lambda_{\theta}^{-1}(x) := \widehat{x}, \quad \widehat{x}(s) = \tau_{\theta}(x U_{\theta}(s)^*), \quad s \in \mathbb{R}^d.
\end{equation}
\end{definition}
Moreover, we have the Plancherel's (Parseval's) identity (see, \cite{MSX}) defined by 
\begin{equation}\label{Plancherel}
\|\lambda_{\theta}(f)\|_{L^{2}(\mathbb{R}^{d}_{\theta})}=\|f\|_{L^{2}(\mathbb{R}^{d})},\quad f\in L^{2}(\mathbb{R}^{d}).
\end{equation}
\begin{definition}\label{F-multiplier}
Let $ g\in\mathcal{S}(\mathbb{R}^d).$ Then the
noncommutative Fourier multiplier $ g(D)$ on $ \mathcal{S}(\mathbb{R}_{\theta}^d)$ with the symbol $g$ is defined by
\begin{equation}\label{Fourier-multiplier}
g(D)(x) = \int_{\mathbb{R}^d} g(s) \widehat{x}(s) U_{\theta}(s) \, ds.
\end{equation}
\end{definition}

The following theorem provides the $ L^p $-$ L^q $ boundedness of the Fourier multiplier on the noncommutative Euclidean space.

\begin{theorem}\label{Mainthm_1}\cite[Theorem 4.3]{Tulenov1}\label{MainTh-1} (H\"ormander Multiplier Theorem) 
Let $ 1 < p \leq 2 \leq q < \infty $. If $ g : \mathbb{R}^d \to \mathbb{C} $ is a measurable function such that
\begin{equation}\label{Symbol_condition}
\sup_{t > 0} t \left(\int\limits_{|g(\xi)| \geq t} d\xi\right)^{\frac{1}{p} - \frac{1}{q}}<\infty,
\end{equation} 
then the Fourier multiplier defined by \eqref{Fourier-multiplier} extends to a bounded map from $ L^p(\mathbb{R}^d_\theta) $ to $ L^q(\mathbb{R}^d_\theta) $. Moreover, we have
\begin{equation}\label{H-multiplier theorem}
\|g(D)\|_{L^p(\mathbb{R}^d_\theta) \rightarrow L^q(\mathbb{R}^d_\theta)} \lesssim \sup_{t > 0} t \left(\int\limits_{|g(\xi)| \geq t} d\xi\right)^{\frac{1}{p} - \frac{1}{q}}.
\end{equation}
\end{theorem}

\section{Applications of $L^p$-$L^q$ boundedness of Fourier multipliers to PDEs on the noncommutative Euclidean space}\label{section-3}

\subsection{$\mathcal{L}$-heat-wave-Schr\"odinger type equations}
\label{sec:main}
Let $0<\alpha<2$ and $1 < p \leq 2 \leq q < \infty.$ In this subsection,  we obtain the $L^p$-$L^q$ norm estimates for the solutions of heat, wave and Schr\"odinger type equations generated by the Fourier multiplier
 $\mathcal{L}$ from $L^{p}(\mathbb{R}^d_{\theta})$ to  $L^{q}(\mathbb{R}^d_{\theta})$  defined by 
\begin{equation}\label{Def-operator}
\mathcal{L}(x):=\int\limits_{\mathbb{R}^d}\sigma(\xi)\widehat{x}(\xi)U_{\theta}(\xi)d\xi 
\end{equation}
with  the symbol $\sigma>0$ satisfying the condition
\begin{equation}\label{Def-M}
M_t:=\sup\limits_{0<\rho<1}\rho\left(\Vol\{\xi\in\mathbb{R}^d:  \sigma(\xi) \leq t^{-\alpha}(\rho^{-1}-1)\}\right)^{\frac{1}{p}-\frac{1}{q}}<\infty, 
\end{equation}
for $t>0.$ 
Here and below, ``$\text{Vol}$" means the volume of a set in $\mathbb{R}^d.$

Let us define the Caputo fractional operator (see, \cite[Chapter 2]{Kilbas}) in time as follows
\begin{equation*} 
  \left(^cD_{t}^{\alpha}f \right)\left( s \right)=\left(I^{\left[ \alpha  \right]-\alpha } \partial^{\left[ \alpha  \right]}f\right)(s), \quad s>0,
  \end{equation*}
where $[\alpha]$ denotes the smallest integer greater or equal than $\alpha>0$ and the  Riemann-Liouville fractional integral $I^{\alpha }f$ of order $\alpha > 0$ (see, \cite[Chapter 2]{Kilbas}) is defined by 
\begin{equation*} 
  \left( I^{\alpha }f \right)\left( s \right)=\frac{1}{ \Gamma\left( \alpha  \right)}\int\limits_{0}^{s} (s-t)^{\alpha-1} f\left( t \right){{d}t}, \quad s>0,  
\end{equation*}
where  $f$ is a continuous function on the interval $[0, s]$.  

\begin{remark} Let $\beta \in \mathbb{R}$ and $ 0<\alpha<2.$   For further investigation, we will use the two-parameter Mittag-Leffler function (see, \cite[Section 1.8]{Kilbas}) frequently
\begin{eqnarray*} 
E_{\alpha, \beta}(z) = \sum_{k=0}^{+\infty} \frac{z^k}{\Gamma(\alpha k + \beta)}, \quad z\in \mathbb{C}.  
\end{eqnarray*}
{\color{red} For the  case  $0<\alpha<2,$ we have the following estimate (see \cite[subsection 2.1]{Podlubny1999}) 
\begin{eqnarray}\label{ML-estimate-2}
|E_{\alpha,\beta}(z)| \leq \frac{C}{1+|z|}, \quad  z\in\mathbb{C},\quad \nu\leq |\arg(z)|\leq \pi,
\end{eqnarray}}
where $\nu$ is a real number such that  $\frac{\pi\alpha}{2}\leq\nu\leq\min\{\pi, \pi\alpha\}$ and $C$ is a positive constant. Note that, $E_{\alpha,1}(\cdot)$ is denoted briefly by $E_{\alpha}(\cdot),$ which is the usual Mittag-Leffler function.
\end{remark}
 For a Banach space $X$ and $0<T\leq \infty$, we denote by $C([0, T), X)$ the Banach space of continuous $X$-valued functions on the interval $[0, T)$ with the norm 
$$
\|f\|_{C([0, T),X)} = \sup\limits_{0 \leq s < T} \|f(s)\|_X.
$$

 \subsubsection{$\mathcal{L}$-heat type equations} We consider the following heat type equation
\begin{equation}\label{Main-equation-1}
\left(^cD_{t}^{\alpha}u \right)(t)+ \mathcal{L}u(t) = 0, \quad t > 0, \quad 0 < \alpha \leq 1, 
\end{equation}
with the initial data
\begin{equation}\label{Initial-date-1}
u(0) = u_0\in L^p(\mathbb{R}^d_{\theta}).   
\end{equation}
The next result gives the $L^p $-$L^q$ estimates for the solution of problem \eqref{Main-equation-1}-\eqref{Initial-date-1}.

\begin{theorem}\label{MainTh-2}
Let $0 < \alpha \leq 1$ and $1 < p \leq 2 \leq q < \infty.$  Let $\mathcal{L}$ be an operator defined by \eqref{Def-operator} with the symbol $\sigma>0$ satisfying condition \eqref{Def-M}.
If $u_0 \in L^p(\mathbb{R}^d_{\theta})$, then there exists a solution $u \in C([0, \infty); L^q(\mathbb{R}^d_{\theta}))$ for the $\mathcal{L}$-heat type equation \eqref{Main-equation-1}-\eqref{Initial-date-1}, represented by
$$
u(t)=E_{\alpha}(-t^\alpha \mathcal{L}) u_0, \quad t > 0,  
$$
where the propagator is defined as
\begin{equation}\label{def_E} 
E_{\alpha}(-t^\alpha \mathcal{L})u_0:=   \int\limits_{\mathbb{R}^{d}}E_{\alpha}(-t^{\alpha}\sigma(\xi) )\widehat{u}_0(\xi)U_{\theta}(\xi)d\xi, \,\ 
E_{\alpha}(-t^\alpha \mathcal{L})u_0:=   \int\limits_{\mathbb{R}^{d}}E_{\alpha}(-t^{\alpha}\sigma(\xi) )\widehat{u}_0(\xi)U_{\theta}(\xi)d\xi.
\end{equation}
Moreover, for all $t > 0$ we have the following estimate
\begin{eqnarray}\label{H-time-decay-rate-1}
\|u(t)\|_{L^q(\mathbb{R}^d_{\theta})} \lesssim M_t \|u_0\|_{L^p(\mathbb{R}^d_{\theta})}.
\end{eqnarray}
In particular, if 
\begin{eqnarray}\label{H-sigms-condition}
\sigma(\xi) \gtrsim  |\xi|^\lambda \quad\text{as} \quad |\xi|\to \infty, \quad \text{for some } \lambda> 0,
\end{eqnarray}
then we get the following time decay rate  for all $t > 0,$
\begin{eqnarray}\label{H-time-decay-rate-2}
\|u(t)\|_{L^q(\mathbb{R}^d_{\theta})} \leq C_{\alpha,\lambda,p,q} t^{-\frac{d\alpha}{\lambda}(\frac{1}{p}-\frac{1}{q})} \|u_0\|_{L^p(\mathbb{R}^d_{\theta})}, \quad \lambda\geq d\left(\frac{1}{p}-\frac{1}{q}\right).
\end{eqnarray}
\end{theorem}
\begin{proof}
Let $\xi \in \mathbb{R}^d$, and denote by $\widehat{u}(t, \xi)$ the Fourier transform of $u(t)$ at $\xi.$   Thus, it follows from equations \eqref{Main-equation-1} and \eqref{Initial-date-1} that
$$ 
\begin{cases}
 \left(^cD_{t}^{\alpha}\widehat{u} \right)(t, \xi) + \sigma(\xi)\widehat{u}(t, \xi) = 0, & t > 0,\\
\widehat{u}(0, \xi) = \widehat{u}_0(\xi).
\end{cases}
$$
Here, we get a system of ordinary differential equations. Thus,   by  \cite[Example 4.9]{Kilbas}  we arrive at
$$
\widehat{u}(t, \xi) = E_{\alpha}(-t^{\alpha}\sigma(\xi)) \widehat{u}_0(\xi).    
$$
Now, we apply the inverse Fourier transform to $\widehat{u}(t,\xi)$ and get the explicit solution as follows
$$
u(t) =\int\limits_{\mathbb{R}^{d}}\widehat{u}(t, \xi)U_{\theta}(\xi)d\xi =\int\limits_{\mathbb{R}^{d}}E_{\alpha}(- t^{\alpha}\sigma(\xi)) \widehat{u}_0(\xi)U_{\theta}(\xi)d\xi =E_{\alpha}(-t^{\alpha}\mathcal{L}) u_0,  
$$
for $0<\alpha\le1.$    Let us prove that the  operator $E_{\alpha}(-t^{\alpha}\mathcal{L})$ is a Fourier multiplier from $L^p(\mathbb{R}^d_{\theta})$ to $L^q(\mathbb{R}^d_{\theta}).$ In order to show that, we first need to  show $E_{\alpha}(-t^{\alpha}\sigma(\cdot) )\in L^{r, \infty} (\mathbb{R}^d),$ $\frac{1}{r}=\frac{1}{p}-\frac{1}{q}.$ It follows from \eqref{ML-estimate-2} and  \eqref{Def-M}  that
\begin{eqnarray}\label{H-estimate}
\|E_{\alpha}(-t^{\alpha}\sigma(\cdot))\|_{L^{r, \infty} (\mathbb{R}^d)}&=&\sup\limits_{\rho>0}\rho\left(\int\limits_{|E_{\alpha}(-t^{\alpha}\sigma(\xi))|\geq \rho}d\xi\right)^{\frac{1}{p}-\frac{1}{q}}\nonumber\\&=& \sup\limits_{ \rho>0}\rho\left(\Vol\{\xi\in\mathbb{R}^d: |E_{\alpha}(-t^{\alpha}\sigma(\xi))|\geq \rho\}\right)^{\frac{1}{p}-\frac{1}{q}}\nonumber\\&\overset{\eqref{ML-estimate-2}}{\leq}& \sup\limits_{ \rho>0}\rho\left(\Vol\{\xi\in\mathbb{R}^d: \frac{1}{1+t^\alpha\sigma(\xi)}\geq \rho\}\right)^{\frac{1}{p}-\frac{1}{q}}\\ &=& \sup\limits_{0<\rho<1}\rho\left(\Vol\{\xi\in\mathbb{R}^d: \sigma(\xi)\leq t^{-\alpha}(\rho^{-1}-1)\}\right)^{\frac{1}{p}-\frac{1}{q}}\nonumber\\
&\overset{\eqref{Def-M}}{=}&M_t<\infty,\quad t>0\nonumber.
\end{eqnarray}
Therefore, by Theorem \ref{MainTh-1},
$E_{\alpha}(-t^{\alpha}\mathcal{L})$ is a Fourier multiplier from $L^p(\mathbb{R}^d_{\theta})$ to $L^q(\mathbb{R}^d_{\theta})$ with the symbol $E_{\alpha}(-t^{\alpha}\sigma(\cdot) ),$ and consequently,  we have
\begin{equation}\label{E1-inequality}
\|u(t)\|_{L^q(\mathbb{R}^d_{\theta})}=\|E_{\alpha}(-t^{\alpha}\mathcal{L}) u_0\|_{L^q(\mathbb{R}^d_{\theta})}\overset{\eqref{H-estimate}}{\lesssim}  
 M_t \|u_0\|_{L^p(\mathbb{R}^d_{\theta})}<\infty, \quad   t>0, 
\end{equation}
which means that  $u \in C([0, \infty); L^q(\mathbb{R}^d_{\theta}))$ and \eqref{H-time-decay-rate-1}   holds. This proves the first part of the theorem.
Now, let us prove the second part. Indeed,
for any $t>0$ by \eqref{H-sigms-condition} and \eqref{H-estimate} we have
\begin{eqnarray}\label{EL-inequality-3} 
M_t
&\overset{\eqref{H-sigms-condition}\eqref{H-estimate}}{\lesssim}&\sup\limits_{0<\rho<1}\rho\left(\Vol\{\xi\in\mathbb{R}^d: |\xi|^\lambda\leq t^{-\alpha}(\rho^{-1}-1)\}\right)^{\frac{1}{p}-\frac{1}{q}}\nonumber\\
&\le&\sup\limits_{0<\rho<1}\rho\left(\Vol\{\xi\in\mathbb{R}^d: |\xi|\leq t^{-\frac{\alpha}{\lambda}} \rho^{-\frac{1}{\lambda}} \}\right)^{\frac{1}{p}-\frac{1}{q}} \\
&\lesssim&t^{-\frac{d\alpha}{\lambda}(\frac{1}{p}-\frac{1}{q})}\sup\limits_{0<\rho<1}\rho^{1-\frac{d}{\lambda}(\frac{1}{p}-\frac{1}{q})}\nonumber\\
&\lesssim&t^{-\frac{d\alpha}{\lambda}(\frac{1}{p}-\frac{1}{q})}\nonumber, \quad  \lambda\geq d\left(\frac{1}{p}-\frac{1}{q}\right).
\end{eqnarray}
Combining \eqref{E1-inequality} and \eqref{EL-inequality-3} we obtain the time decay rate \eqref{H-time-decay-rate-2}. 
This completes the proof.
\end{proof} 
 
\subsubsection{$\mathcal{L}$-Schr\"odinger type equation}  
In this section we consider the $\mathcal{L}$-Schr\"odinger type equation
\begin{equation}\label{Main-equation-2}
i\left(^cD_{t}^{\alpha}u \right)(t) + \mathcal{L}u(t) = 0, \quad t > 0, \quad 0 < \alpha \le 1, 
\end{equation}
with the initial data
\begin{equation}\label{Initial-date-2}
u(0) = u_0 \in L^p(\mathbb{R}^d_{\theta}),   
\end{equation}
where   the fractional  derivative of Caputo type $^cD_{t}^{\alpha}$ is of order $\alpha > 0.$  
\begin{theorem}\label{MainTh-3}
Let $0 < \alpha \leq 1$ and $1 < p \leq 2 \leq q < \infty.$    Let  $\mathcal{L}$ be an operator defined by \eqref{Def-operator} with the symbol $\sigma>0$ satisfying condition \eqref{Def-M}. If $u_0 \in L^p(\mathbb{R}^d_{\theta})$, then there exists a solution $u \in C([0, \infty); L^q(\mathbb{R}^d_{\theta}))$ for the $\mathcal{L}$-Schr\"odinger type equation \eqref{Main-equation-2}-\eqref{Initial-date-2}, represented by
$$
u(t) = E_{\alpha}(it^\alpha \mathcal{L}) u_0, \quad t > 0, 
$$
where 
$$E_{\alpha}(it^{\alpha}\mathcal{L}) u_0:=   \int\limits_{\mathbb{R}^{d}}E_{\alpha}(it^{\alpha}\sigma(\xi) )\widehat{u}_0(\xi)U_{\theta}(\xi)d\xi.$$
Moreover, for all $t > 0$ we have the following estimate
\begin{equation}\label{S-time-decay-rate-1}
\|u(t)\|_{L^q(\mathbb{R}^d_{\theta})} \lesssim M_t \|u_0\|_{L^p(\mathbb{R}^d_{\theta})}.
\end{equation}
In particular, if 
\begin{equation*}
\sigma(\xi) \gtrsim  |\xi|^\lambda \quad\text{as} \quad |\xi|\to \infty, \quad \text{for some } \lambda> 0,
\end{equation*}
then we get the following time decay rate  for all $t > 0,$
\begin{equation}\label{S-time-decay-rate-2}
\|u(t)\|_{L^q(\mathbb{R}^d_{\theta})} \leq C_{\alpha,\lambda,p,q} t^{-\frac{d\alpha}{\lambda}(\frac{1}{p} - \frac{1}{q})} \|u_0\|_{L^p(\mathbb{R}^d_{\theta})}, \quad \lambda\geq d\left(\frac{1}{p}-\frac{1}{q}\right).
\end{equation}
\end{theorem}
\begin{proof} 
Observe that the solution to the problem   \eqref{Main-equation-2}-\eqref{Initial-date-2} is expressed as
 \begin{equation}\label{E-operator-2} 
u(t)=   \int\limits_{\mathbb{R}^{d}}E_{\alpha}(it^{\alpha}\sigma(\xi) )\widehat{u}_0(\xi)U_{\theta}(\xi)d\xi, \quad  0<\alpha\le1. 
\end{equation}
Now we check that the right hand side of \eqref{E-operator-2} is a Fourier multiplier from $L^p(\mathbb{R}^d_{\theta})$ to $L^q(\mathbb{R}^d_{\theta}).$ 
It follows from  \eqref{Def-M} and \eqref{ML-estimate-2} that     
\begin{eqnarray}\label{S-estimate}
\|E_{\alpha}(it^{\alpha}\sigma(\cdot))\|_{L^{r, \infty} (\mathbb{R}^d)}&=&\sup\limits_{\rho>0}\rho\left(\int\limits_{|E_{\alpha}(it^{\alpha}\sigma(\xi))|\geq \rho}d\xi\right)^{\frac{1}{p}-\frac{1}{q}}\nonumber\\&=& \sup\limits_{ \rho>0}\rho\left(\Vol\{\xi\in\mathbb{R}^d: |E_{\alpha}(it^{\alpha}\sigma(\xi))|\geq \rho\}\right)^{\frac{1}{p}-\frac{1}{q}}\nonumber\\&\overset{\eqref{ML-estimate-2}}{\leq}& \sup\limits_{ \rho>0}\rho\left(\Vol\{\xi\in\mathbb{R}^d: \frac{1}{1+t^\alpha\sigma(\xi)}\geq \rho\}\right)^{\frac{1}{p}-\frac{1}{q}}\\ &=& \sup\limits_{0<\rho<1}\rho\left(\Vol\{\xi\in\mathbb{R}^d: \sigma(\xi)\leq t^{-\alpha}(\rho^{-1}-1)\}\right)^{\frac{1}{p}-\frac{1}{q}}\nonumber\\
&\overset{\eqref{Def-M}}{=}&M_t<\infty,\quad t>0\nonumber.
\end{eqnarray}
Thus, by Theorem \ref{MainTh-1}, the  operator $E_{\alpha}(it^{\alpha}\mathcal{L})$ is a Fourier multiplier from $L^p(\mathbb{R}^d_{\theta})$ to $L^q(\mathbb{R}^d_{\theta})$ with the symbol $E_{\alpha}(it^{\alpha}\sigma(\cdot) ).$  It follows from  \eqref{S-estimate}   that
\begin{equation}\label{ES-inequality-1}
\|u(t)\|_{L^q(\mathbb{R}^d_{\theta})} = \|E_{\alpha}(it^{\alpha}\mathcal{L}) u_0\|_{L^q(\mathbb{R}^d_{\theta})} \overset{\eqref{S-estimate}}{\lesssim} M_t \|u_0\|_{L^p(\mathbb{R}^d_{\theta})}, \quad t>0.
\end{equation}
Consequently, we obtain that $u \in C([0, \infty); L^q(\mathbb{R}^d_{\theta}))$ and \eqref{S-time-decay-rate-1}. Hence,  by \eqref{EL-inequality-3} and \eqref{ES-inequality-1} we obtain the time decay rate \eqref{S-time-decay-rate-2} for all $t > 0.$   This completes the proof. 
\end{proof}
 
\subsubsection{$\mathcal{L}$-wave type equation} 
In this subsection, we examine the solution to the following equation
\begin{eqnarray}\label{Main-equation-3}
\left(^cD_{t}^{\alpha}u \right)(t)+ \mathcal{L}u(t) = 0, \quad t > 0,\quad 1 < \alpha < 2,    
\end{eqnarray}
with initial data 
\begin{eqnarray}\label{Initial-date-3}
u(0)= u_0 \in L^p(\mathbb{R}^d_{\theta})
\end{eqnarray}
and 
\begin{eqnarray}\label{Initial-date-4}
\partial_t u(0) = u_1 \in L^p(\mathbb{R}^d_{\theta}),
\end{eqnarray}
where   the fractional  derivative of Caputo type $^cD_{t}^{\alpha}$ is of order $\alpha > 0.$  
\begin{theorem}\label{wave-thm-4} Let $1<\alpha<2$ and $1 < p \leq 2 \leq q < \infty.$  Let $\mathcal{L}$ be an operator defined by \eqref{Def-operator} with a symbol $\sigma>0$ satisfying condition \eqref{Def-M}. If $u_0, u_1 \in L^p(\mathbb{R}^d_{\theta}),$ then there exists a solution $u\in C([0,\infty);L^q(\mathbb{R}^d_{\theta}))$ for the $\mathcal{L}$-wave type equation \eqref{Main-equation-3}-\eqref{Initial-date-4} which is given by
$$
u(t) = E_{\alpha}(-t^{\alpha}\mathcal{L}) u_0+tE_{\alpha,2}(-t^{\alpha}\mathcal{L}) u_1, \quad t > 0, 
$$
where  the operator  $E_{\alpha}(-t^{\alpha}\mathcal{L}) u_0$ defined by \eqref{def_E} and 
$$
E_{\alpha,2}(-\xi^{\alpha}\mathcal{L})u_1:=\int\limits_{\mathbb{R}^{d}}E_{\alpha,2}(- t^{\alpha}\sigma(\xi))\widehat{u}_1(\xi)U_{\theta}(\xi)d\xi, \quad t > 0.  
$$
Moreover, for any $t>0$
we have the following estimate
\begin{equation}\label{W-time-decay-rate-1}
\|u(t)\|_{L^q(\mathbb{R}^d_{\theta})} \lesssim M_t \left(\|u_0\|_{L^p(\mathbb{R}^d_{\theta})} + t\|u_1\|_{L^p(\mathbb{R}^d_{\theta})}\right), \quad t>0.
\end{equation}
In particular, if  
\begin{equation*}
\sigma(\xi)\gtrsim  |\xi|^\lambda \quad\text{as} \quad |\xi|\to \infty, \quad \text{for some } \lambda> 0,
\end{equation*}
then we get the following time decay rate for $t>0$
\begin{equation}\label{W-time-decay-rate-2}
\|u(t)\|_{L^q(\mathbb{R}^d_{\theta})} \leq C_{\alpha,\lambda,p,q} t^{-\frac{d\alpha}{\lambda} \left(\frac{1}{p} - \frac{1}{q}\right)} \left(\|u_0\|_{L^p(\mathbb{R}^d_{\theta})} + t\|u_1\|_{L^p(\mathbb{R}^d_{\theta})}\right), \quad \lambda\geq d\left(\frac{1}{p}-\frac{1}{q}\right).
\end{equation}
\end{theorem}
\begin{proof} Let $\xi\in \mathbb{R}^d$, and denote by $\widehat{u}(t, \xi)$ the  Fourier transform of $u(t)$ with respect to the variable $\xi$. Hence, it follows from equations \eqref{Main-equation-3}-\eqref{Initial-date-3} that
\begin{equation}\label{ODE}
\begin{cases}
\left(^cD_{t}^{\alpha}\widehat{u} \right)(t, \xi) + \sigma(\xi)\widehat{u}(t, \xi) = 0, & t > 0,\\
\widehat{u}(0, \xi) = \widehat{u}_0(\xi),\\
\partial_t \widehat{u}(0, \xi) = \widehat{u}_1(\xi). 
\end{cases}
\end{equation}
Here, we also obtain a system of ordinary differential equations as the previous subsection. Thus, by \cite[Example 4.10]{Kilbas}, we arrive at
$$
\widehat{u}(t, \xi) = E_{\alpha,1}(-t^{\alpha}\sigma(\xi)) \widehat{u}_0(\xi)+tE_{\alpha,2}(-t^{\alpha}\sigma(\xi)) \widehat{u}_1(\xi).
$$
We now perform the inverse Fourier transform on  $\mathbb{R}^d_{\theta}$ to derive the explicit solution to the problem as follows:
\begin{eqnarray}\label{W-solution} 
u(t) &=&  \int\limits_{\mathbb{R}^{d}}\widehat{u}(t, \xi)U_{\theta}(\xi)d\xi\nonumber\\
&=&\int\limits_{\mathbb{R}^{d}}E_{\alpha,1}(- t^{\alpha}\sigma(\xi)) \widehat{u}_0(\xi)U_{\theta}(\xi)d\xi+t\int\limits_{\mathbb{R}^{d}}E_{\alpha,2}(- t^{\alpha}\sigma(\xi))\widehat{u}_1(\xi)U_{\theta}(\xi)d\xi\\
&=&E_{\alpha}(-t^{\alpha}\mathcal{L}) u_0+tE_{\alpha,2}(-t^{\alpha}\mathcal{L}) u_1. \nonumber
\end{eqnarray}
Now, let us show  that $E_{\alpha,2}(-t^{\alpha}\mathcal{L})$ is Fourier multiplier from $L^p(\mathbb{R}^d_{\theta})$ to $L^q(\mathbb{R}^d_{\theta}).$
 Using  \eqref{Def-M} and \eqref{ML-estimate-2} we have    
\begin{eqnarray}\label{W-estimate}
\|E_{\alpha,2}(-t^{\alpha}\sigma(\cdot))\|_{L^{r, \infty} (\mathbb{R}^d)}&=&\sup\limits_{\rho>0}\rho\left(\int\limits_{|E_{\alpha,2}(-t^{\alpha}\sigma(\xi))|\geq \rho}d\xi\right)^{\frac{1}{p}-\frac{1}{q}}\nonumber\\&=& \sup\limits_{ \rho>0}\rho\left(\Vol\{\xi\in\mathbb{R}^d: |E_{\alpha,2}(-t^{\alpha}\sigma(\xi))|\geq \rho\}\right)^{\frac{1}{p}-\frac{1}{q}}\nonumber\\&\overset{\eqref{ML-estimate-2}}{\leq}& \sup\limits_{ \rho>0}\rho\left(\Vol\{\xi\in\mathbb{R}^d: \frac{1}{1+t^\alpha\sigma(\xi)}\geq \rho\}\right)^{\frac{1}{p}-\frac{1}{q}}\\ &=& \sup\limits_{0<\rho<1}\rho\left(\Vol\{\xi\in\mathbb{R}^d: \sigma(\xi)\leq t^{-\alpha}(\rho^{-1}-1)\}\right)^{\frac{1}{p}-\frac{1}{q}}\nonumber\\
&\overset{\eqref{Def-M}}{=}&M_t<\infty,\quad t>0.\nonumber
\end{eqnarray}
Hence, by Theorem \ref{MainTh-1}  the  operator $E_{\alpha,2}(-t^{\alpha}\mathcal{L})$ is a Fourier multiplier from $L^p(\mathbb{R}^d_{\theta})$ to $L^q(\mathbb{R}^d_{\theta})$ with the symbol $E_{\alpha,2}(-t^{\alpha}\sigma(\cdot) ).$  Next, it follows from  \eqref{H-estimate} and \eqref{W-solution}-\eqref{W-estimate}    that  
\begin{eqnarray} \label{W-inequality}
\|u(t)\|_{L^q(\mathbb{R}^d_{\theta})}&\overset{\eqref{W-solution}}{\leq}&\|E_{\alpha}(-t^{\alpha}\mathcal{L}) u_0 \|_{L^q(\mathbb{R}^d_{\theta})} + t\|E_{\alpha,2}(-t^{\alpha}\mathcal{L}) u_1\|_{L^q(\mathbb{R}^d_{\theta})} \nonumber\\
 &\overset{\eqref{H-estimate}\eqref{W-estimate}}{\lesssim}& M_t \left(\|u_0\|_{L^p(\mathbb{R}^d_{\theta})} +  t  \|u_1\|_{L^p(\mathbb{R}^d_{\theta})}\right), \quad t>0.  
\end{eqnarray}
Therefore, the solution $u\in C([0,\infty);L^q(\mathbb{R}^d_{\theta}))$ and \eqref{W-time-decay-rate-1}. By  \eqref{EL-inequality-3}  and \eqref{W-inequality} we have the time decay rate \eqref{W-time-decay-rate-2} for all $t > 0.$  This completes the proof. 
\end{proof}
 
\subsection{Example} 
We introduce the group of translations $\{T_s\}_{s\in\mathbb{R}^d},$ where $T_s$ is presented as the unique $\ast$-automorphism of $L^{\infty}(\mathbb{R}^d_\theta)$ that operates on $U_{\theta}(t)$ as follows:
\begin{equation}\label{def-translations-1}
T_s(U_{\theta}(t))=e^{i(t,s)}U_{\theta}(t), \quad  
t,s \in\mathbb{R}^d, 
\end{equation}
where $(\cdot,\cdot)$ denotes the standard inner product in $\mathbb{R}^d.$  Then partial derivations $\partial^{\theta}_j,  j = 1, \dots, d,$ are defined on smooth elements $x$ by 
$$
\partial^{\theta}_j(x) = \frac{d}{ds_{j}}T_s(x)|_{s=0}. 
$$
By \eqref{def-translations-1} and \eqref{def-integration}, it is easily verified that 
$$
\partial^{\theta}_j(x)\overset{\eqref{def-integration}}{=}\partial^{\theta}_j\lambda_{\theta}(f)\overset{\eqref{def-translations-1}}{=}\lambda_{\theta}(it_{j}f(t)), \quad  j=1,\cdots,d, \,\  f\in \mathcal{S}(\mathbb{R}^d),
$$
for $x=\lambda_{\theta}(f).$  
For a multi-index $\alpha=(\alpha_1,...,\alpha_d),$ we define 
$$
\partial^{\alpha}_{\theta}=(\partial^{\theta}_{1})^{\alpha_{1}}\dots(\partial^{\theta}_{d})^{\alpha_{d}}, 
$$
and the gradient $\nabla_{\theta}$ associated with $\mathbb{R}^{d}_{\theta}$ is the operator
$$
\nabla_{\theta}=(\partial^{\theta}_{1},  \dots,  \partial^{\theta}_{d}).  
$$
Moreover, the Laplace operator $\Delta_{\theta}$ is defined as
\begin{equation}\label{laplacian}
\Delta_{\theta} = (\partial_1^{\theta})^2   + \cdots +(\partial_d^{\theta})^2. 
\end  {equation}
Note that $-\Delta_{\theta}$ is a positive operator on $L^2(\mathbb{R}^{d}_\theta)$ with the spectrum $\{ |\xi|^2  : \xi \in \mathbb{R}^d\}$ (see \cite{MSX} and \cite{Mc}). {\color{red} Let $\mathcal{L}= (-\Delta_{\theta})^\frac{s}{2}$ with the symbol  $\sigma(\xi):=|\xi|^s$ and $s>0.$    Then a unique solution of the $\mathcal{L}$-heat type equation \eqref{Main-equation-1}-\eqref{Initial-date-1}, represented by
$$
u(t)=E_{\alpha}(-t^\alpha(-\Delta_{\theta})^\frac{s}{2}) u_0, \quad t > 0.  
$$ 
Moreover,
\begin{eqnarray}\label{E-condition} 
\|E_{\alpha}(-t^{\alpha}\sigma(\cdot))\|_{L^{r, \infty} (\mathbb{R}^d)}&{\le}&\sup\limits_{0<\rho<1}\rho\left(\Vol\{\xi\in\mathbb{R}^d: |\xi|^s\leq t^{-\alpha}(\rho^{-1}-1)\}\right)^{\frac{1}{p}-\frac{1}{q}}\nonumber\\
&\le&\sup\limits_{0<\rho<1}\rho\left(\Vol\{\xi\in\mathbb{R}^d: |\xi|\leq t^{-\frac{\alpha}{s}} \rho^{-\frac{1}{s}} \}\right)^{\frac{1}{p}-\frac{1}{q}}\\
&\lesssim&t^{-\frac{d\alpha}{s}(\frac{1}{p}-\frac{1}{q})}\sup\limits_{0<\rho<1}\rho^{1-\frac{d}{s}(\frac{1}{p}-\frac{1}{q})}\nonumber\\
&\lesssim&t^{-\frac{d\alpha}{s}(\frac{1}{p}-\frac{1}{q})}\nonumber,  
\end{eqnarray}
whenever $p\neq q$ such that $2\leq d\leq \frac{s}{\left(\frac{1}{p}-\frac{1}{q}\right)}.$ Therefore, it satisfies conditions of Theorems \ref{MainTh-1} and \eqref{E-condition}  we get the following time decay rate  for all $t > 0,$
\begin{eqnarray}\label{S-time-decay-rate}
\|u(t)\|_{L^q(\mathbb{R}^d_{\theta})} \leq C_{\alpha,p,q} t^{-\frac{d\alpha}{s}(\frac{1}{p} - \frac{1}{q})} \|u_0\|_{L^p(\mathbb{R}^d_{\theta})}, \quad 2\leq d\leq \frac{s}{\left(\frac{1}{p}-\frac{1}{q}\right)}, \,\ p\neq q.
\end{eqnarray} }

\section{Nonlinear PDEs}
In this subsection, we explore various applications of Theorem \ref{MainTh-1} in the context of certain nonlinear PDEs on noncommutative Euclidean spaces. Our main goal is to demonstrate the well-posedness properties of these nonlinear equations in this setting. The core idea is based on the studies in \cite{CKRT} and \cite{KR}, where analogous problems were investigated on smooth manifolds.

\subsection{Nonlinear heat equation}
First, let us consider the following Cauchy problem: 
 \begin{equation}\label{Eq4_1}
 \begin{split}
     & \partial_t u(t)-h(t)|\mathcal{A}u(t)|^{p}=0, \\ 
     & u(0)=u_{0}.
 \end{split}
      \end{equation}
Here, we write $\partial_t$ for differentiation in the time variable $t,$  $h$ is a positive bounded function on $(0,T],$ and $\mathcal{A}$ is a linear bounded operator from $L^{2}(\mathbb{R}^{d}_{\theta})$ to $L^{2p}(\mathbb{R}^{d}_{\theta}),$ $1\leq p<\infty.$ 
 Here, also note that $|\mathcal{A}u(t)|^{p}:=((\mathcal{A}u(t))^*\mathcal{A}u(t))^{p/2}$ and $(\mathcal{A}u(t))^*$ is the adjoint operator of $\mathcal{A}u(t).$

For a Banach space $X$, we denote by $L^\infty((0, T], X)$ the Banach space of essentially bounded $X$-valued functions on the interval $(0, T]$ with the norm 
$$ 
 \|u\|_{L^\infty((0,T]; X)} = \operatorname{ess\,sup}\limits\limits_{t \in (0,T]} \|u(t)\|_{X} < \infty. 
 $$

 Our purpose is to study the well-posedness of the heat equation \eqref{Eq4_1}.
 \begin{definition} We say that the  Cauchy problem \eqref{Eq4_1} admits a solution $u$ if it satisfies 
 \begin{equation}\label{Eq4_2}
u(t)=u_{0}+\int\limits_{0}^{t}h(s)|\mathcal{A}u(s)|^{p}ds
 \end{equation}
in the space $L^{\infty}((0,T]; L^{2}(\mathbb{R}^{d}_{\theta}))$ for every $T<\infty$. 

We say that the Cauchy problem \eqref{Eq4_1} admits a local solution $u$ if it satisfies equation \eqref{Eq4_2} in space $L^{\infty}((0,T^{*}]; L^{2}(\mathbb{R}^{d}_{\theta}))$ for some $T^{*}>0$.
 \end{definition}
 \begin{theorem}\label{Mainthm_2} Let $p\geq2$ integer number and $u_0\in L^{2}(\mathbb{R}^{d}_{\theta}).$  Suppose that $\mathcal{A}$ is a linear bounded operator from $L^{2}(\mathbb{R}^{d}_{\theta})$ to $L^{2p}(\mathbb{R}^{d}_{\theta}).$  
Then the Cauchy problem \eqref{Eq4_1} has a local solution in the space $L^{\infty}((0,T^{*}]; L^{2}(\mathbb{R}^{d}_{\theta}))$ for some $T^{*}>0$.
 \end{theorem}
\begin{proof} First, by integrating equation \eqref{Eq4_1} with respect to $t$, we get
\begin{equation*}
u(t)=u_{0}+\int\limits_{0}^{t}h(s)|\mathcal{A}u(s)|^pds.
\end{equation*}
Let $w(\cdot):=|\mathcal{A}u(\cdot)|^p.$ Then, using \eqref{direct-F-transform} we derive 
\begin{equation}\label{Eq4_3}
\widehat{u}(t,\xi)=\widehat{u}_{0}(\xi)+\int\limits_{0}^{t}h(s)\widehat{w}(s,\xi)ds,\quad \xi\in\mathbb{R}^{d}, 
\end{equation}
where $\widehat{u}(t, \xi)$ and $\widehat{w}(t, \xi)$ denote the Fourier transforms of the operators $u(t)$ and $w(t)$, respectively, with respect to the variable $\xi \in \mathbb{R}^d$. Thus, by the Plancherel identity \eqref{Plancherel} and \eqref{Eq4_3} we obtain the following estimate 
\begin{equation}\label{Eq4_4} 
\|u(t)\|^{2}_{L^{2}(\mathbb{R}^{d}_{\theta})}\overset{\eqref{Plancherel}}{=}\|\widehat{u}(t)\|^{2}_{L^{2}(\mathbb{R}^{d})}\leq c\left(  \|\widehat{u}_{0}\|^{2}_{L^{2}(\mathbb{R}^{d})}  +  \int\limits_{\mathbb{R}^{d}}\left|\int\limits_{0}^{t}h(s)\widehat{w}(s,\xi)ds \right|^{2}d\xi \right). 
\end{equation}
By the Cauchy-Schwarz inequality we have
\begin{eqnarray}\label{Eq4_5} 
\int\limits_{\mathbb{R}^{d}}\left|\int\limits_{0}^{t}h(s)\widehat{w}(s, \xi)ds \right|^{2}d\xi&\leq&\int\limits_{\mathbb{R}^{d}}\int\limits_{0}^{t}\left|h(s)\right|^{2}ds \int\limits_{0}^{t}\left|\widehat{w}(s, \xi)\right|^{2}ds d\xi\nonumber\\
&=&\|h\|^2_{L^2(0,T)} \int\limits_{0}^{t}\|\widehat{w}(s, \xi)\|_{L^{2}(\mathbb{R}^{d})}ds. 
\end{eqnarray}
By combining \eqref{Eq4_4} and \eqref{Eq4_5}, using Plancherel identity \eqref{Plancherel} we obtain
 \begin{eqnarray}\label{Eq4_6} 
     \|u(t)\|^{2}_{L^{2}(\mathbb{R}^{d}_{\theta})}&\overset{\eqref{Eq4_5}}{\leq}&c\left(  \|\widehat{u}_{0}\|^{2}_{L^{2}(\mathbb{R}^{d})}  +  \|h\|^2_{L^2(0,T)} \int\limits_{0}^{t}\|\widehat{w}(s, \xi)\|_{L^{2}(\mathbb{R}^{d})}ds \right)\nonumber\\&\overset{\eqref{Plancherel}}{=}& c \left(  \|u_{0}\|^{2}_{L^{2}(\mathbb{R}^{d}_{\theta})}  +    \|h\|^2_{L^2(0,T)} \int\limits_{0}^{t}\|w(s)\|_{L^{2}(\mathbb{R}^{d}_{\theta})}ds \right) \nonumber\\
     &=& c \left(  \|u_{0}\|^{2}_{L^{2}(\mathbb{R}^{d}_{\theta})}  + \|h\|^2_{L^2(0,T)}\int\limits_{0}^{t}\tau_{\theta}(|\mathcal{A}u(s)|^{2p})ds \right)\\
     &=& c \left( \|u_{0}\|^{2}_{L^{2}(\mathbb{R}^{d}_{\theta})}  + \|h\|^2_{L^2(0,T)}\int\limits_{0}^{t}\|\mathcal{A}u(s)\|^{2p}_{L^{2p}(\mathbb{R}^{d}_{\theta})}ds \right).\nonumber 
 \end{eqnarray}
Since the operator $\mathcal{A}$ is bounded from $L^{2}(\mathbb{R}^{d}_{\theta})$ to $L^{2p}(\mathbb{R}^{d}_{\theta})$, there exit a constant $m>0$ such that 
\begin{equation}\label{Eq4_7}
\|\mathcal{A}u\|_{L^{2p}(\mathbb{R}^{d}_{\theta})}\leq m\|u\|_{L^{2}(\mathbb{R}^{d}_{\theta})}, \quad u\in L^{2}(\mathbb{R}^{d}_{\theta}). 
\end{equation}
Hence,  applying \eqref{Eq4_7}, we can express \eqref{Eq4_6} in the following form  
  \begin{equation}\label{Eq4_8} 
     \|u(t)\|^{2}_{L^{2}(\mathbb{R}^{d}_{\theta})} \overset{\eqref{Eq4_7}}{\leq} c \left(  \|u_{0}\|^{2}_{L^{2}(\mathbb{R}^{d}_{\theta})}  + \|h\|^2_{L^2(0,T)}\int\limits_{0}^{t}\|u(s)\|^{2p}_{L^{2}(\mathbb{R}^{d}_{\theta})}ds \right). 
 \end{equation}
By taking the $L^{\infty}$-norm on both sides of the estimate \eqref{Eq4_8}, one obtain 
 \begin{equation}\label{Eq4_9} 
     \|u\|^{2}_{L^{\infty}((0,T];L^{2}(\mathbb{R}^{d}_{\theta}))}  \overset{\eqref{Eq4_8}}{\leq} c \left(  \|u_{0}\|^{2}_{L^{2}(\mathbb{R}^{d}_{\theta})}  + T \|h\|^2_{L^2(0,T)}\|u\|^{2p}_{L^{\infty}((0,T];L^{2}(\mathbb{R}^{d}_{\theta}))} \right). 
 \end{equation}
 Let us denote the following set by  $\mathcal{X}:$ 
\begin{equation*}
\mathcal{X}:=\{u\in L^{\infty}((0,T];L^{2}(\mathbb{R}^{d}_{\theta})) : \|u\|_{ L^{\infty}((0,T];L^{2}(\mathbb{R}^{d}_{\theta}))}\leq \delta \|u_{0}\|_{L^{2}(\mathbb{R}^{d}_{\theta})}\}
\end{equation*}
for some constant $\delta\geq1$. Next, for $u \in \mathcal{X},$ we define the map
$$
(\mathcal{K} u)(t) = u_{0}+\int\limits_{0}^{t}h(s)|\mathcal{A}u(s)|^{p}ds, \quad 0 < t \leq T.
$$
We are going to prove there exists a unique local solution as a fixed point of $\mathcal{K}$ by the Banach fixed point theorem. First, let us show that $\mathcal{K}$ maps $\mathcal{X}$ into itself.  Indeed, if $u \in \mathcal{X}.$ Then, by \eqref{Eq4_8} we have  
\begin{eqnarray*}
\|(\mathcal{K} u)(t)\|^2_{L^2(\mathbb{R}^d_{\theta})}&\overset{\eqref{Eq4_8}}{\leq}& c \left(  \|u_{0}\|^{2}_{L^{2}(\mathbb{R}^{d}_{\theta})}  + T \|h\|^2_{L^2(0,T)}\|u\|^{2p}_{L^{\infty}((0,T];L^{2}(\mathbb{R}^{d}_{\theta}))} \right)\\
&\leq& c \left(\|u_{0}\|^{2}_{L^{2}(\mathbb{R}^{d}_{\theta}))}+T \|h\|^2_{L^2(0,T)} \delta^{2p}\|u_0\|^{2p}_{ L^{2}(\mathbb{R}^{d}_{\theta})}\right)\\
&\leq& c \left(1+T \|h\|^2_{L^2(0,T)} \delta^{2p}\|u_0\|^{2(p-1)}_{ L^{2}(\mathbb{R}^{d}_{\theta})}\right)\delta\|u_{0}\|^{2}_{L^{2}(\mathbb{R}^{d}_{\theta}))}. 
\end{eqnarray*}
Choose $T > 0$  such  that 
\begin{equation*}
    T\leq T^{*}:=\frac{ \sqrt{c^{2}-1}} {\|h\|_{L^2(0,T)} \delta^{p}\|u_0\|^{p-1}_{ L^{2}(\mathbb{R}^{d}_{\theta})}}.
\end{equation*}
Consequently,  
$$
\|\mathcal{K} (u)\|_{L^{\infty}((0,T];L^{2}(\mathbb{R}^{d}_{\theta}))}\leq \delta \|u_{0}\|_{L^{2}(\mathbb{R}^{d}_{\theta})},
$$
which shows that $\mathcal{K}(u)\in \mathcal{X}.$   Let us now show that $\mathcal{K}$ is a contraction map.  By Theorem 2.2 (i) in \cite{DDdPS}  and the linearity of $\mathcal{A}$  we obtain
\begin{eqnarray}\label{Eq4_10}
\||\mathcal{A}(u)|-|\mathcal{A}(v)|\|_{L^{2p}(\mathbb{R}^d_{\theta})}&\leq&c  \|\mathcal{A}(u)-\mathcal{A}(v))\|_{L^{2p}(\mathbb{R}^d_{\theta})} \nonumber\\&=& c  \|\mathcal{A}(u-v)\|_{L^{2p}(\mathbb{R}^d_{\theta})} \\&\overset{\eqref{Eq4_6}}{\leq}& c m \|u-v\|_{L^{2}(\mathbb{R}^d_{\theta})}.\nonumber  
\end{eqnarray}
It has been shown in \cite[Section 5.3]{RST2} that the operator function $f(t)=t^{p}$ satisfies the following inequality:
\begin{equation}\label{Eq4_11}
\|u^p-v^p\|_{L^2(\mathbb{R}^d_{\theta})}\leq c_3 (\|u\|^{p-1}_{L^{2p}(\mathbb{R}^d_{\theta})}+\|v\|^{p-1}_{L^{2p}(\mathbb{R}^d_{\theta})})\|u-v\|_{L^{2p}(\mathbb{R}^d_{\theta})} 
\end{equation}
 for all positive operators $u,v\in L^{2p}(\mathbb{R}^d_{\theta}).$ Thus, combining  \eqref{Eq4_6}, \eqref{Eq4_10} and \eqref{Eq4_11},  we obtain
\begin{eqnarray}\label{Eq4_12}
\||\mathcal{A}(u)|^p-|\mathcal{A}(v)|^p\|_{L^2(\mathbb{R}^d_{\theta})}&\overset{\eqref{Eq4_11}}{\leq}& c_3 (\|\mathcal{A}(u)\|^{p-1}_{L^{2p}(\mathbb{R}^d_{\theta})}+\|\mathcal{A}(v)\|^{p-1}_{L^{2p}(\mathbb{R}^d_{\theta})})\||\mathcal{A}u|-|\mathcal{A}v|\|_{L^{2p}(\mathbb{R}^d_{\theta})}\nonumber\\
&\overset{\eqref{Eq4_6}}{\leq}& 2c_3m^{p-1} (\|u\|^{p-1}_{L^{2}(\mathbb{R}^d_{\theta})}+\|v\|^{p-1}_{L^{2}(\mathbb{R}^d_{\theta})})\||\mathcal{A}u|-|\mathcal{A}v|\|_{L^{2p}(\mathbb{R}^d_{\theta})}\\
&\overset{\eqref{Eq4_10}}{\leq}& 2c_3c_4m^{p} \delta^{p-1} \|u_0\|^{p-1}_{L^{2}(\mathbb{R}^d_{\theta})} \| u- v\|_{L^{2}(\mathbb{R}^d_{\theta})},\nonumber 
\end{eqnarray}
for $u,v\in \mathcal{X}.$ Let $T_1$ be chosen such that $T_1:=T <\{2c_2c_4m^{p} \delta^{p-1} \|u_0\|^{p-1}_{L^{2}(\mathbb{R}^d_{\theta})}\}^{-1}.$ Then, by \eqref{Eq4_12} we obtain  
  \begin{eqnarray} \label{Eq4_13}
    \|(\mathcal{K}u)(t) - (\mathcal{K}v)(t)\|_{L^2(\mathbb{R}^d_{\theta})} 
    &\leq& \int\limits_{0}^{t} \||\mathcal{A}u(s)|^{p}- |\mathcal{A}v(s)|^{p}\|_{L^2(\mathbb{R}^d_{\theta})}ds \nonumber\\
    &\overset{\eqref{Eq4_12}}{\leq}&2c_2c_4m^{p} \delta^{p-1} \|u_0\|^{p-1}_{L^{2}(\mathbb{R}^d_{\theta})}\int\limits_{0}^{t}  \|u(s) -  v(s)\|_{L^{2}(\mathbb{R}^d_{\theta})}\\
    &\le&T2c_2c_4m^{p} \delta^{p-1} \|u_0\|^{p-1}_{L^{2}(\mathbb{R}^d_{\theta})} \|u-  v\|_{L^{\infty}((0,T];L^{2}(\mathbb{R}^{d}_{\theta}))},\nonumber
\end{eqnarray}
for all $T\leq T_1.$ Hence, taking $L^{\infty}$-norm from the both sides of  \eqref{Eq4_13} with respect to $T\leq \min\{T_0, T_1\},$ for any $u,v\in\mathcal{X}$ we obtain
\begin{eqnarray*}
    \|\mathcal{K}(u) - \mathcal{K}(v)\|_{L^{\infty}((0,T];L^{2}(\mathbb{R}^{d}_{\theta}))} 
    \leq \|u - v\|_{L^{\infty}((0,T];L^{2}(\mathbb{R}^{d}_{\theta}))}, 
\end{eqnarray*}
which shows that $\mathcal{K}$ is contraction on $\mathcal{X}.$ Hence, Banach fixed point theorem implies that $\mathcal{K}$ has a unique
fixed point $u\in  \mathcal{X}$ such that $u(t)=(\mathcal{K} u)(t).$ It thus shows that problem \eqref{Eq4_1} has a unique solution $u\in \mathcal{X}\subset L^{\infty}((0,T]; L^2(\mathbb{R}^d_{\theta}))$ for all $T\leq \min\{T_0, T_1\}.$  

About the uniqueness of the solution, let $u$ and $v$ be two solutions of \eqref{Eq4_1}. We define
$$
h(t) = \|u(t) - v(t)\|_{L^2(\mathbb{R}^d_{\theta})}, \quad  t \geq 0,
$$
and repeating the argument in  \eqref{Eq4_13}, we have
$$
h(t)\overset{\eqref{Eq4_2}}{\leq}\int\limits_{0}^{t} \||\mathcal{A}u(s)|^{p}- |\mathcal{A}v(s)|^{p}\|_{L^2(\mathbb{R}^d_{\theta})}ds \leq c \int\limits_0^t h(s)ds, \quad t \geq 0.
$$
Therefore, by the Gronwall lemma, we obtain that $h(t) \equiv 0$, which implies $u(t) \equiv v(t)$ for all $t \geq 0$.

\end{proof} 
As a result of Theorem \ref{Mainthm_1} and Theorem \ref{Mainthm_2}, we arrive at the following noteworthy conclusion.
 \begin{theorem} Let $p\geq2$ integer number and $u_0\in L^{2}(\mathbb{R}^{d}_{\theta}).$ Suppose that $\mathcal{A}$ is Fourier multiplier such that its symbol $\sigma$ satisfies the condition \eqref{Symbol_condition}. Then the  Cauchy problem \eqref{Eq4_1} has a local solution in the space $L^{\infty}((0,T^{*}]; L^{2}(\mathbb{R}^{d}_{\theta}))$ for some $T^{*}>0$.
 \end{theorem}
\begin{proof} Given that $\varphi \in L^{r,\infty}(\mathbb{R}^{d})$ with $1/r = 1/2 - 1/2p.$ Thus, Theorem \ref{Mainthm_1} implies that $\mathcal{A}$ is bounded from $L^{2}(\mathbb{R}^{d}_{\theta})$ to $L^{2p}(\mathbb{R}^{d}_{\theta})$. Consequently, the statement follows directly from Theorem \ref{Mainthm_2}.
\end{proof}

\subsection{Nonlinear wave equation}
Now, we examine the Cauchy problem for the nonlinear wave equation in $L^{\infty}((0,T]; L^{2}(\mathbb{R}^{d}_{\theta}))$:
\begin{equation}\label{Eq4_14}
    \begin{split}
        & \partial^{2}_t u(t)-h(t)|\mathcal{A}u(t)|^{p}=0,\\
        & u(0)=u_{0}, \quad \partial_t u(0)=u_{1}.
    \end{split}
\end{equation}
Here, we assume that $u_0, u_1\in L^{2}(\mathbb{R}^{d}_{\theta}).$ We write $\partial_t$ for differentiation in the time variable $t,$ $h$ is a positive bounded function on $(0,T]$ and $\mathcal{A}$ is a linear bounded operator from $L^{2}(\mathbb{R}^{d}_{\theta})$ to $L^{2p}(\mathbb{R}^{d}_{\theta}),$ and $p\geq2$ integer number. 
\begin{definition} We say that the initial value problem \eqref{Eq4_14} admits a global solution $u$ if it satisfies 
\begin{equation}\label{Eq4_15}
        u(t)=u_{0}+tu_{1}+\int\limits_{0}^{t}(t-s)h(s)|\mathcal{A}u(s)|^{p}ds
\end{equation}
in the space $L^{\infty}((0,T]; L^{2}(\mathbb{R}^{d}_{\theta}))$ for every $T>0$.

We say that \eqref{Eq4_14} admits a local solution $u$ if it satisfies the equation \eqref{Eq4_15} in the space $L^{\infty}((0,T^{*}]; L^{2}(\mathbb{R}^{d}_{\theta}))$ for some $T^{*}>0$.
\end{definition}

\begin{theorem}\label{wave-thm} Let $p\geq2$ integer number. Suppose that $\mathcal{A}$ is a linear bounded operator from $L^{2}(\mathbb{R}^{d}_{\theta})$ to $L^{2p}(\mathbb{R}^{d}_{\theta}).$  
\begin{itemize}
    \item [(i)] If $\|h\|_{L^{2}(0,T)}<\infty$ for some $T>0$ and $u_0, u_1 \in L^{2}(\mathbb{R}^{d}_{\theta})$, then the Cauchy problem \eqref{Eq4_14} has a local solution in $L^{\infty}((0,T]; L^{2}(\mathbb{R}^{d}_{\theta}))$.
    \item [(ii)] Suppose that $u_{1}$ is identically equal to zero. Let $\gamma > 3/2$. Moreover, assume that $\|h\|_{L^{2}(0,T)} \leq c T^{-\gamma}$ for every $T > 0$, where $c$ does not depend on $T$. Then, for every $T > 0$, the   Cauchy problem \eqref{Eq4_14} has a solution in the space  $L^{\infty}((0,T]; L^{2}(\mathbb{R}^{d}_{\theta}))$ for all $u_{0}\in L^{2}(\mathbb{R}^{d}_{\theta})$ with a sufficiently small norm.
\end{itemize}
\end{theorem}
\begin{proof} First, we prove (i). By integrating equation \eqref{Eq4_14} twice with respect to  $t,$ one obtains
\begin{equation*}
     u(t)=u_{0}+tu_{1}+\int\limits_{0}^{t}(t-s)h(s)|\mathcal{A}u(s)|^{p}ds.
\end{equation*}
We denote $w(\cdot):=|\mathcal{A}u(\cdot)|^{p}.$ Then, using \eqref{direct-F-transform} we derive 
\begin{equation}\label{Eq4_16}
\widehat{u}(t)=\widehat{u}_{0}(\xi)+t\widehat{u}_{1}(\xi)+\int\limits_{0}^{t}(t-s)h(s)\widehat{w}(s,\xi)ds, \quad \xi \in \mathbb{R}^d,
\end{equation}
where the functions $\widehat{u}(t, \xi)$ and $\widehat{w}(t, \xi)$  are the Fourier transform of $u(t)$ with respect to the variable $\xi \in \mathbb{R}^d.$ Hence, applying the Plancherel identity \eqref{Plancherel}  with \eqref{Eq4_16}, we obtain the following estimate
\begin{eqnarray}\label{Eq4_17}  
 \|u(t)\|_{L^{2}(\mathbb{R}^{d}_{\theta})}^{2}&=&\|\widehat{u}(t)\|_{L^{2}(\mathbb{R}^{d})}^{2}\nonumber\\
 &\leq& c\left(\|\widehat{u}_{0}\|_{L^{2}(\mathbb{R}^{d})}^{2} + t^{2}\|\widehat{u}_{1}\|_{L^{2}(\mathbb{R}^{d})}^{2} + \| \int\limits_{0}^{t}(t-s)h(s) \widehat{w}(s,\cdot)ds\|_{L^{2}(\mathbb{R}^{d})}^{2} \right)    
\end{eqnarray}
By applying the same argument as in   \eqref{Eq4_6}, we obtain
\begin{equation}\label{Eq4_18} 
\| \int\limits_{0}^{t}(t-s)h(s)\widehat{w}(s,\cdot)ds\|_{L^{2}(\mathbb{R}^{d})}^{2} \leq  t^2 \|h\|^2_{L^2(0,T)}   \int\limits_{0}^{t} \|\widehat{w}(s, \xi)\|_{L^{2}(\mathbb{R}^{d})}^{2}ds.   
\end{equation}
Next, by combining \eqref{Eq4_6}, \eqref{Eq4_17} and \eqref{Eq4_18} and applying the Plancherel identity \eqref{Plancherel}, we can deduce the following estimate  
\begin{eqnarray}\label{Eq4_19}      
\|u(t)\|_{L^{2}(\mathbb{R}^{d}_{\theta})}^{2}   
& \overset{\eqref{Eq4_18}}{\leq}& c \left( \|\widehat{u}_{0}\|_{L^{2}(\mathbb{R}^{d})}^{2} + t^{2}\|\widehat{u}_{1}\|_{L^{2}(\mathbb{R}^{d})}^{2} +t^2 \|h\|^2_{L^2(0,T)}   \int\limits_{0}^{t} \|\widehat{w}(s, \xi)\|_{L^{2}(\mathbb{R}^{d})}^{2}ds \right)\nonumber\\
&=& c \left(  \|u_{0}\|_{L^{2}(\mathbb{R}^{d}_{\theta})}^{2} + t^{2}\|u_{1}\|_{L^{2}(\mathbb{R}^{d}_{\theta})}^{2} +  t^{2}\|h\|^{2}_{L^{2}(0,T)}\int\limits_{0}^{t}\|\mathcal{A}u(s)\|^{2p}_{L^{2}(\mathbb{R}^{d}_{\theta})}ds  \right)\\
&\leq& c\left(  \|u_{0}\|_{L^{2}(\mathbb{R}^{d}_{\theta})}^{2} + t^{2}\|u_{1}\|_{L^{2}(\mathbb{R}^{d}_{\theta})}^{2} +  t^{2}\|h\|^{2}_{L^{2}(0,T)}\int\limits_{0}^{t}\|u(s)\|^{2p}_{L^{2}(\mathbb{R}^{d}_{\theta})}ds\right).\nonumber
\end{eqnarray}
 Thus, taking the  $L^{\infty}$-norm on both sides of the  estimate \eqref{Eq4_19} yields
\begin{equation}\label{Eq4_20}
\|u\|^{2}_{L^{\infty}((0,T];L^{2}(\mathbb{R}^{d}_{\theta}))}  \overset{\eqref{Eq4_19}}{\leq} c \left(  \|u_{0}\|_{L^{2}(\mathbb{R}^{d}_{\theta})}^{2} + T^{2}\|u_{1}\|_{L^{2}(\mathbb{R}^{d}_{\theta})}^{2}\right.  
 +\left.  T^{3}\|h\|^{2}_{L^{2}(0,T)}\|u\|^{2p}_{L^{\infty}((0,T];L^{2}(\mathbb{R}^{d}_{\theta}))}  \right).
\end{equation}
Let us set $\mathcal{X}_{1}$ by  
\begin{equation*}
    \mathcal{X}_{1}:=\left\{u\in L^{\infty}((0,T];L^{2}(\mathbb{R}^{d}_{\theta})) : \|u\|_{ L^{\infty}((0,T];L^{2}(\mathbb{R}^{d}_{\theta}))}\leq \delta_1 \left(\|u_{0}\|_{L^{2}(\mathbb{R}^{d}_{\theta})}+T\|u_{1}\|_{L^{2}(\mathbb{R}^{d}_{\theta})}\right)\right\},
\end{equation*}
for some constant $\delta_1>1$.  For $u \in \Omega,$ we define the map
$$
(\mathcal{K}_1 u)(t) = u_{0}+tu_{1}+\int\limits_{0}^{t}(t-s)h(s)|\mathcal{A}u(s)|^{p}ds, \quad 0 < t \leq T.
$$
We aim to establish the existence of a unique local solution as the fixed point of $\mathcal{K}_1$ using the Banach fixed point theorem. To begin, we will demonstrate that $\mathcal{K}_1$ maps $\Omega$ into itself. For $u\in \mathcal{X}_{1},$ by \eqref{Eq4_20} we derive 
\begin{eqnarray*} 
\|(\mathcal{K}_1 u)(t)\|_{L^{2}(\mathbb{R}^{d}_{\theta})}^{2}&\overset{\eqref{Eq4_20}}{\leq}&c\left(\|u_{0}\|_{L^{2}(\mathbb{R}^{d}_{\theta})}^{2} + T^{2}\|u_{1}\|_{L^{2}(\mathbb{R}^{d}_{\theta})}^{2} +  T^{3}\|h\|^{2}_{L^{2}(0,T)}\|u\|^{2}_{L^{\infty}((0,T];L^{2}(\mathbb{R}^{d}_{\theta}))}\right)\\
     & \leq& c\left(\|u_{0}\|_{L^{2}(\mathbb{R}^{d}_{\theta})}^{2} + T^{2}\|u_{1}\|_{L^{2}(\mathbb{R}^{d}_{\theta})}^{2} \right.\\
     &+&  \left.T^{3}\|h\|^{2}_{L^{2}(0,T)}\delta_{1}^{2p}\left[ \|u_{0}\|^{2}_{L^{2}(\mathbb{R}^{d}_{\theta})}+T^{2}\|u_{1}\|^{2}_{L^{2}(\mathbb{R}^{d}_{\theta})}\right]^{p}\right). 
\end{eqnarray*}
Choose $T > 0$  such  that 
\begin{equation*}
    T\leq T^{*}:=\min\left[\left(\frac{c_{1}-1}{\|g\|^{2}_{L^{2}(0,T)}\delta^{p-1}\|u_{0}\|^{2p-2}_{L^{2}(\mathbb{R}^{d}_{\theta})}}\right)^{\frac{1}{3}},\left(\frac{c_{1}-1}{\|g\|^{2}_{L^{2}(0,T)}\delta^{p-1}\|u_{1}\|^{2p-2}_{L^{2}(\mathbb{R}^{d}_{\theta})}}\right)^{\frac{1}{3}}\right].
\end{equation*} Consequently,  
$$
\|\mathcal{K}_1 (u)\|_{L^{\infty}((0,T];L^{2}(\mathbb{R}^{d}_{\theta}))}\leq \delta \|u_{0}\|_{L^{2}(\mathbb{R}^{d}_{\theta})},
$$
which shows that $\mathcal{K}_1(u)\in \mathcal{X}_1.$

Next, we demonstrate that $\mathcal{K}_1$ is a contraction map.  By repeating the argument used in \eqref{Eq4_13}, we have 
\begin{eqnarray}\label{Eq5_10}
\||\mathcal{A}(u)|^p-|\mathcal{A}(v)|^p\|_{L^2(\mathbb{R}^d_{\theta})}&\overset{\eqref{Eq4_13}}{\leq}&  2c_2m^{p} (\|u\|^{p-1}_{L^{2}(\mathbb{R}^d_{\theta})}+\|v\|^{p-1}_{L^{2}(\mathbb{R}^d_{\theta})})\| u- v\|_{L^{2}(\mathbb{R}^d_{\theta})}\nonumber \\
&{\leq}&  4c_2m^{p} \delta^{p-1} \left(\|u_{0}\|^2_{L^{2}(\mathbb{R}^{d}_{\theta})}+T^2\|u_{1}\|^2_{L^{2}(\mathbb{R}^{d}_{\theta})}\right)\| u- v\|_{L^{2}(\mathbb{R}^d_{\theta})},
\end{eqnarray}
for $u,v \in \Omega_1$. Let $T_2$ be chosen such that 
$$
T_2 := T < \left\{4c_2m^{p} \delta^{p-1} \left(\|u_{0}\|^2_{L^{2}(\mathbb{R}^{d}_{\theta})}+T^2\|u_{1}\|^2_{L^{2}(\mathbb{R}^{d}_{\theta})}\right)\right\}^{-1}.
$$
Then, from \eqref{Eq5_10}, we derive 
\begin{equation}\label{Eq5_11}
\|(\mathcal{K}_1u)(t) - (\mathcal{K}_1v)(t)\|_{L^2(\mathbb{R}^d_{\theta})}\leq\int\limits_{0}^{t} \||\mathcal{A}u(s)|^{p}- |\mathcal{A}v(s)|^{p}\|_{L^2(\mathbb{R}^d_{\theta})}ds\overset{\eqref{Eq5_10}}{\leq}\| u- v\|_{L^{2}(\mathbb{R}^d_{\theta})},
\end{equation}
for all $T \leq T_2$. Taking the $L^{\infty}$-norm on both sides of   \eqref{Eq5_11} with respect to $T \leq \min\{T_0, T_2\}$, we obtain 
\begin{eqnarray*}
    \|\mathcal{K}_1(u) - \mathcal{K}_1(v)\|_{L^{\infty}((0,T];L^{2}(\mathbb{R}^{d}_{\theta}))} 
    \leq \|u - v\|_{L^{\infty}((0,T];L^{2}(\mathbb{R}^{d}_{\theta}))}, \quad u,v \in \Omega_1, 
\end{eqnarray*}
proving that $\mathcal{K}_1$ is a contraction on $\Omega_1$. By the Banach Fixed Point Theorem, $\mathcal{K}_1$ has a unique fixed point $u \in \Omega_1$ such that $u(t) = (\mathcal{K}_1u)(t)$. It follows that the problem \eqref{Eq4_14} has a unique solution $u \in \Omega_1 \subset L^{\infty}((0,T];L^2(\mathbb{R}^d_{\theta}))$ for all $T \leq \min\{T_0, T_2\}$. The proof of the uniqueness of the solution  follows from the same approach as in Theorem \ref{Mainthm_2}.

We now need to prove part (ii). To begin, using the same reasoning as in the proof of part (i), we arrive at \eqref{Eq4_20}. Thus, based on our assumptions, we obtain
 \begin{equation*}
     \|u\|^{2}_{L^{\infty}((0,T];L^{2}(\mathbb{R}^{d}_{\theta}))}\leq C(   \|u_{0}\|_{L^{2}(\mathbb{R}^{d}_{\theta})}^{2}+ T^{3-2\gamma} \|u\|_{L^{\infty}((0,T]; L^{2}(\mathbb{R}^{d}_{\theta})}^{2p}).
 \end{equation*}
 For fixed constant $c_2> 1$, let us introduce by $\mathcal{X}_{c_{2}}$ the following set
\begin{equation}
    \mathcal{X}_{c_{2}}:=\left\{u\in L^{\infty}((0,T];L^{2}(\mathbb{R}^{d}_{\theta})) : \|u\|^{2}_{L^{\infty}((0,T];L^{2}(\mathbb{R}^{d}_{\theta}))}\leq c_{2}T^{\gamma_{0}}\|u_{0}\|^{2}_{L^{2}(\mathbb{R}^{d}_{\theta})} \right\},
\end{equation}
where  a constant  $\gamma_{0}>0$ will be defined later.  Now, for $u\in \mathcal{X}_{c_{2}},$ it holds that
\begin{equation*}
    \|u_{0}\|^{2}_{L^{2}(\mathbb{R}^{d}_{\theta})} +T^{3-2\gamma}\|u\|^{2p}_{L^{\infty}((0,T];L^{2}(\mathbb{R}^{d}_{\theta}))}\leq   \|u_{0}\|^{2}_{L^{2}(\mathbb{R}^{d}_{\theta})} +T^{3-2\gamma+\gamma_{0}p}c^{p}\|u\|^{2p}_{L^{2}(\mathbb{R}^{d}_{\theta})}.
\end{equation*}
To ensure that $u\in \mathcal{X}_{c_{2}},$ it is necessary that
\begin{equation*}
 \|u_{0}\|^{2}_{L^{2}(\mathbb{R}^{d}_{\theta})} +T^{3-2\gamma+\gamma_{0}p}c^{p}\|u\|^{2p}_{L^{2}(\mathbb{R}^{d}_{\theta})}\leq c_{2}T^{\gamma_{0}}\|u\|^{2p}_{L^{2}(\mathbb{R}^{d}_{\theta})}.
\end{equation*}
Thus, by choosing $0<\gamma_{0}<\frac{2\gamma-3}{p}$ such that
\begin{equation*}
    \Tilde{\gamma}:=3-2\gamma+\gamma_{0}p<0,
\end{equation*}
we obtain
 \begin{equation*}
   c^{p}\|u_{0}\|^{2p-2}_{L^{2}(\mathbb{R}^{d}_{\theta})} \leq c_{2} T^{-\Tilde{\gamma}+\gamma_{0}}. 
\end{equation*}
Based on the previous estimate, we conclude that for any $T > 0$, there exists a sufficiently small $\|u_{0}\|_{L^{2}(\mathbb{R}^{d}_{\theta})}$ such that the initial value problem \eqref{Eq4_14} admits a solution. This establishes part (ii) of the Theorem \ref{wave-thm}.
\end{proof}
\begin{cor}  Let $p\geq2$ integer number. Suppose that $\mathcal{A}$ is a Fourier multiplier such that its symbol $\sigma$ satisfies  condition \eqref{Symbol_condition}.
\begin{itemize}
\item [(i)] If $\|h\|_{L^{2}(0,T)}<\infty$ for some $T>0$, then the Cauchy problem \eqref{Eq4_14} has a local solution in $L^{\infty}((0,T]; L^{2}(\mathbb{R}^{d}_{\theta}))$.
\item [(ii)] Suppose that $u_{1}$ is identically equal to zero. Let $\gamma > 3/2$. Moreover, assume that $\|h\|_{L^{2}(0,T)} \leq c T^{-\gamma}$ for every $T > 0$, where $c$ does not depend on $T$. Then, for every $T > 0$, the wave equation (or Cauchy problem) \eqref{Eq4_14} has a solution in the space  $L^{\infty}((0,T]; L^{2}(\mathbb{R}^{d}_{\theta}))$ for all $u_{0}\in L^{2}(\mathbb{R}^{d}_{\theta})$ with a sufficiently small norm.
\end{itemize}
\end{cor}

\section{Acknowledgements}
The work was partially supported by the grant from the Ministry of Science and Higher Education of the Republic of Kazakhstan (Grant No. AP23487088). 
The authors were also partially supported by Odysseus and Methusalem grants (01M01021 (BOF Methusalem) and 3G0H9418 (FWO Odysseus)) from Ghent Analysis and PDE center at Ghent University. The first author was also supported by the EPSRC grant EP/V005529/1. 
Authors thank the anonymous referee for reading the paper and providing thoughtful comments, which improved the exposition of the paper.

\begin{center}

\end{center}


\begin{thebibliography}{999} 

    \bibitem{AR} R.~Akylzhanov, M.~Ruzhansky. {\it $L_p$-$L_q$ multipliers on locally compact groups}. J. Funct. Anal., 278(3) (2020), 49, 108324. 

    \bibitem{BW} A.~Barchielli, R.F.~Werner. {\it Hybrid quantum-classical systems: quasi-free Markovian dynamics.} Int. J. Quantum Inf., 22(5):Paper No. 2440002, 51, 2024.
     
    \bibitem{BB} C.~Bastos, O.~Bertolami, Nuno Costa Dias, Joa\'{o} Nuno Prata. {\it Weyl-Wigner formulation of noncommutative quantum mechanics.} J. Math. Phys., 49(7):072101, 24, 2008.
    
    \bibitem{CGRS} A.L.~Carey, V.~Gayral, A.~Rennie, and F.~Sukochev. {\it Index theory for locally compact noncommutative geometries.} Mem. Amer. Math. Soc., 231(1085):vi+130, 2014.
    
    
    \bibitem{CKRT} D.~Cardona, V.~Kumar, M.~Ruzhansky and N.~Tokmagambetov. \emph{$L_{p}-L_{q}$ boundedness of pseudo-differential operators on smooth manifolds and its applications to nonlinear equations.} (2020). 1-37.  arXiv: 2005.04936.
      
   \bibitem{SJR1} S.G.~Cobos, J.E.~Restrepo, M.~Ruzhansky. \emph{ $L^p$-$L^q$ estimates for non-local heat and wave type equations on locally compact groups.} (2024). To appear in C.R. Acad. Sci. Paris.
    
   \bibitem{SJR2} S.G.~Cobos, J.~E.~Restrepo, M.~Ruzhansky. \emph{ Heat-wave-Schr\"odinger type equations on locally compact groups.}  arXiv.2302.00721. 
    
    \bibitem{DW} L.~Dammeier, Reinhard F.~Werner. Quantum-classical hybrid systems and their quasifree transformations. Quantum, 7:1068, 2023.
    
    \bibitem{DPS} P.G.~Dodds, B.~de Pagter and F.A.~Sukochev. {\it Noncommutative Integration and Operator Theory}. Publisher Name Springer, Berlin, Heidelberg. 2023. 

    \bibitem{DDdPS} P.G.~Dodds,  T.K.~Dodds,  B.~de Pagter and   F.A.~Sukochev, {\it Lipschitz continuity of the absolute value and Riesz projections in symmetric operator spaces.} J. Funct. Anal. 148 (1997), no. 1, 28-69.
    
    \bibitem{DLMS} M.~Do\"{r}fler, F.~Luef, H.~McNulty, and E.~Skrettingland. Time-frequency analysis and coorbit spaces of operators. J. Math. Anal. Appl., 534(2):Paper No. 128058, 33, 2024.
    \bibitem{FK} T.~Fack, H.~Kosaki.  {\it Generalized $s$-numbers of $\tau$-measurable operators}, Pacific J. Math., {\bf 123}:2 (1986), 269--300.

       
    \bibitem{GJM} L.~Gao, M.~Junge, and E.~McDonald. {\it Quantum Euclidean spaces with noncommutative derivatives.}  J. Noncommut. Geom. 16 (2022), 153–213.

     \bibitem{GGSVM} V.~Gayral, J.M.~Gracia-Bondia, B.~Iochum, T.~Sch\"ucker, and J.C.~Varilly. {\it Moyal planes are spectral triples.} Comm. Math. Phys., 246(3):569--623, 2004.

     
    \bibitem{GJP} A.M.~Gonz\'{a}lez-P\'{e}rez, M. Junge, and J. Parcet. {\it Singular integrals in quantum Euclidean spaces.}  Mem. Amer. Math. Soc. 272 (2021), no. 1334, xiii+90 pp.


    \bibitem{GJP2} A.M.~Gonz\'{a}lez-P\'{e}rez, M. Junge, and J. Parcet. Smooth Fourier multipliers in group algebras via Sobolev dimension. Ann. Sci. E \'{c}. Norm. Sup\'{e}r. (4), 50(4):879–925, 2017.
          
    \bibitem{green-book} J.M.~Gracia-Bondia, J.~V\'{a}rilly and H.~Figueroa.  {\it  Elements of noncommutative geometry,}  Birkhäuser Boston, Inc., Boston, MA, 2001. xviii+685 pp.
    
        
    \bibitem{Grafakos}
    L.~Grafakos. \emph{Classical Fourier Analysis}, vol. 249 of Graduate Texts in Mathematics, 3rd edn. Springer, New York (2014)

    \bibitem{Gro} H.J.~Groenewold. {\it On the principles of elementary quantum mechanics}. Physica, 12:405--460, 1946.

    \bibitem{H} B. C. Hall. Quantum theory for mathematicians, volume 267 of Graduate Texts in Mathematics. Springer, New York, 2013.
    
    \bibitem{Hor}
    L.~H\"ormander. \emph{Estimates for translation invariant operators in $L_p$ spaces}. Acta Math. \textbf{104} (1960), 93–140.


    \bibitem{KSZ}  J.~Kemppainen, J.~Siljander and  R.~Zacher. {\it Representation of solutions and large-time behavior for fully nonlocal diffusion equations.} J. Differential Equations. 263 (2017) 149-201.

    \bibitem{Kilbas} A.A.~Kilbas. Theory and applications of fractional differential equations. 1st Edition, Volume 204-January 12, 2006. 
    

    \bibitem{KR}  V.~Kumar and M.~Ruzhansky, {\it $L_p-L_q$ boundedness of $(k,a)-$Fourier multipliers with applications to nonlinear equations.} Int. Math. Res. Not. IMRN, 2023, no. 2, 1073-1093. 

    \bibitem{L} L.~Lafleche. On quantum Sobolev inequalities. J. Funct. Anal., 286(10):Paper No. 110400, 40, 2024.

    \bibitem{LSZ} S.~Lord, F.~Sukochev, D.~Zanin. {\it Singular traces. Theory and applications.} De Gruyter Studies in Mathematics, {\bf 46}. De Gruyter, Berlin, 2013.



    \bibitem{MSX}E.~McDonald, F.~Sukochev, and X.~Xiong. {\it Quantum differentiability on noncommutative Euclidean spaces.}  Comm. Math. Phys., 379(2):491--542, 2020.

    
    \bibitem{Mc}E.~McDonald. {\it Nonlinear partial differential equations on noncommutative Euclidean spaces.} J. Evol. Equ. 24:16 (2024), 1--58. 
    
       
    \bibitem{M} J.E.~Moyal. {\it Quantum mechanics as a statistical theory.} Proc. Cambridge Philos. Soc., 45:99-124, 1949.

           
    \bibitem{PXu} G.~Pisier and Q.~Xu. {\it Non-commutative $L_p$-spaces.}  In Handbook of the geometry of Banach spaces, Vol. 2, pages 1459–1517. North-Holland, Amsterdam, 2003.

    \bibitem{Podlubny1999} I.~Podlubny,  Fractional Diferential Equations. Academic Press, San Diego (1999).
    
    \bibitem{XSM2} Zh.~Quanguo  and L.~Yaning. {\it Global well-posedness and blow-up solutions of the Cauchy problem for a time-fractional superdiffusion equation.} J. Evol. Equ. https://doi.org/10.1007/s00028-018-0475-x.

    \bibitem{QH} Zh.~Quan-Guo and S.~Hong-Rui. {\it The blow-up and global existence of solutions of Cauchy problems for a time fractional diffusion equation.} Topol. Methods Nonlinear Anal. 46(1): 69-92 (2015). DOI: 10.12775/TMNA.2015.038 
    

    \bibitem{Rieffel} M.~Rieffel.  {\it Deformation quantization for actions of $\mathbb{R}^d.$} Mem. Amer. Math. Soc. {\bf 106}:506 (1993), x+93 pp.
         
    \bibitem{Tulenov1}
    M.~Ruzhansky, S.~Shaimardan, K.~Tulenov. \emph{$L^p$-$L^q$ boundedness of Fourier multipliers on quantum Euclidean spaces}. 2024. 1--27. arXiv:2312.00657.

 \bibitem{R} J.~Rosenberg. \emph{Noncommutative variations on Laplace's equation}. Anal. PDE, 1(1):95--114, 2008.

    \bibitem{RST2}
    M.~Ruzhansky, S.~Shaimardan, K.~Tulenov. \emph{Sobolev inequality and its applications to nonlinear PDE on noncommutative Euclidean spaces}. (2024), 1-27. arXiv:2408.09100

  

        \bibitem{W} R.~Werner. Quantum harmonic analysis on phase space. J. Math. Phys., 25(5):1404–1411, 1984.
        
      \bibitem{Zhao} M.~Zhao. {\it Smoothing estimates for non commutative spaces}. PhD thesis, University of Illinois at Urbana-Champaign, 2018. http://hdl.handle.net/2142/101570.

      \bibitem{XSM} S.~Xiaoyan, Zh.~Shiliang and L.~Miao.   {\it Local well-posedness of semilinear space-time fractional Schr\"{o}dinger equation} J. Math. Anal. Appl. 479 (2019) 1244–1265. 

     

     
   

\end{thebibliography}
  \end{document}